\documentclass{siamart220329}

\usepackage{lipsum}
\usepackage{amsfonts}
\usepackage{graphicx}
\usepackage{epstopdf}
\usepackage{algorithmic}
\ifpdf
  \DeclareGraphicsExtensions{.eps,.pdf,.png,.jpg}
\else
  \DeclareGraphicsExtensions{.eps}
\fi

\usepackage{amsmath}
\usepackage{amsfonts}
\usepackage{amssymb}
\usepackage{graphicx}
\usepackage{color}
\usepackage{bm}
\usepackage{hyperref}

\newcommand{\norm}[2]{\|{#1}\|_{{#2}}}
\newcommand{\tnorm}[1]{\|\kern-.4mm| {#1} |\kern-.4mm\|}
\newcommand{\ud}{\,{\rm d}}
\newcommand{\jump}[1]{[\kern-0.6mm [{#1}]\kern-0.6mm]}
\newcommand{\tjump}[1]{\lfloor\kern-0.4mm{#1}\kern-0.4mm\rfloor}
\newcommand{\av}[1]{\{\kern-1.5mm \{{#1}\}\kern-1.5mm\}}

\newcommand{\mbf}[1]{\bm{{#1}}}

\newcommand{\Ainner}[2]{(\!({#1},{#2})\!)_A}

\newcommand{\ndg}[1]{|\kern-.05cm \|{#1}|\kern-.05cm \|}

\newcommand{\man}[1]{{\color{black}{#1}}}

\newcommand{\rev}[1]{{\color{black}{#1}}}

\newsiamremark{remark}{Remark}
\newsiamremark{hypothesis}{Hypothesis}
\crefname{hypothesis}{Hypothesis}{Hypotheses}
\newsiamthm{claim}{Claim}

\headers{Hypocoercivity-exploiting SFEM for Kolmogorov equation}{Z. Dong, E. H. Georgoulis, and P. J. Herbert}

\title{A hypocoercivity-exploiting stabilised \\ finite element method for Kolmogorov equation\thanks{Submitted to the editors DATE.
\funding{This work was funded by the EPSRC (grant number EP/W005840/1).}}}

\author{ Zhaonan Dong\thanks{1) Inria, 48 rue Barrault, 75647 Paris, France,
				2) CERMICS, Ecole des Ponts, 77455 Marne-la-Vall\'{e}e, France   (\email{zhaonan.dong@inria.fr}).}
\and Emmanuil H. Georgoulis\thanks{1) Maxwell Institute for Mathematical Sciences and Department of Mathematics, School of Mathematical and Computer Sciences, Heriot-Watt University,   Edinburgh EH14 4AS, United Kingdom (\email{E.Georgoulis@hw.ac.uk}), 2) Department of Mathematics, School of Applied Mathematical and Physical Sciences, National Technical University of Athens, Zografou 15780, Greece \& 3) IACM-FORTH, Greece.}
\and Philip J.~Herbert\thanks{Department of Mathematics, University of Sussex, Brighton BN1 9RF, United Kingdom.
  (\email{P.Herbert@sussex.ac.uk})}
}

\usepackage{amsopn}

\ifpdf
\hypersetup{
  pdftitle={Hypocoercivity-exploiting SFEM for Kolmogorov equation},
  pdfauthor={Z. Dong, E. H. Georgoulis, and P. J. Herbert}
}
\fi

\begin{document}

\maketitle

\begin{abstract}
	We propose a new stabilised finite element method for the classical Kolmogorov equation. The latter serves as a basic model problem for large classes of kinetic-type equations  and, crucially, is characterised by degenerate diffusion. The stabilisation is constructed so that the resulting  method admits a \emph{numerical hypocoercivity} property, analogous to the corresponding property of the PDE problem. More specifically, the stabilisation is constructed so that spectral gap is possible in the resulting ``stronger-than-energy'' stabilisation norm, despite the degenerate nature of the diffusion in Kolmogorov, thereby the method has a provably robust behaviour as the ``time'' variable goes to infinity. We consider both a spatially discrete version of the stabilised finite element method and a fully discrete version, with the time discretisation realised by discontinuous Galerkin timestepping. Both stability and a priori error bounds are proven in all cases. Numerical experiments verify the theoretical findings.
\end{abstract}



\section{Introduction}
In the study of kinetic modelling, degenerate evolution problems are abundant.
These may arise from the integration of stochastic processes, which upon integration typically give rise to Kolmogorov/Fokker-Planck type partial differential equations (PDEs). In many examples, diffusion (stemming from Brownian motion/randomness)  is present in some of the spatial directions in the respective Kolmogorov/Fokker-Planck PDE.
Quite interestingly, despite this, many such PDEs with degenerate diffusion give rise to initial/boundary value problems that are, nonetheless, characterised by convergence to long time equilibria. Villani in \cite{villani} coined the term \emph{hypocoercivity} to signify the subtle property of certain degenerate evolution (integro)differenetial operators to yield dissipation in directions where \emph{no} diffusion is explicitly present. Hypocoercivity has been formulated as a general framework in proving decay to equilibrium for kinetic equations in the seminal work \cite{villani} by Villani, who conceptualised and expanded key ideas proposed in the seminal works of H\'erau \& Nier in  \cite{herau_nier},  Eckmann \& Hairer in \cite{jeckmann_hairer},  Mouhot \& Neumann \cite{mouhot_neumann} and others.

To highlight the role of degeneracy of diffusion with respect to long-time behaviour, we consider the classical Kolmogorov equation:
\begin{equation}\label{eq:kolmogorov}
	\mathcal{L}u\equiv u_t-u_{xx}+xu_y = f,\quad \text{in }(0,t_f]\times \man{\Omega};
\end{equation}
for some final time $t_f>0$, with $\Omega\subset\mathbb{R}^2$, and a suitably smooth forcing $f\colon(0,\infty)\times \Omega\to \mathbb{R}$, along with appropriate initial and boundary data.
(The above notation used above is typical in PDE literature; in kinetic applications, the operator $\mathcal{L}$ may appear as $\mathcal{L}f:=f_t-f_{vv}+vf_x$ with $v$ denoting the particle velocity variable, $x$ the displacement/position and $f$ the respective probability density function.) Let $(\cdot,\cdot)_{L_2(\Omega)}$ the standard $L_2(\Omega)$-inner product, for some 'spatial' domain $\Omega\subset \mathbb{R}^2$, along with the respective norm $\norm{\cdot}{L_2(\Omega)}$.
Assuming suitable boundary conditions, such that the skew-symmetric term $(xu_y,u)$ vanishes, the standard energy argument for \eqref{eq:kolmogorov} yields
\begin{equation}\label{eq:energy_classical}
	\frac{1}{2}\frac{\ud}{\ud t}\norm{u(t)}{L_2(\Omega)}^2+\norm{u_x}{L_2(\Omega)}^2 = (f,u)_{L_2(\Omega)}\le \frac{1}{2{\epsilon}}\norm{f}{L_2(\Omega)}^2+ \frac{{\epsilon}}{2}\norm{u}{L_2(\Omega)}^2,
\end{equation}
for almost all $t\in(0,t_f]$ and all $\epsilon >0$, where $t_f\in\mathbb{R}_+$ is some final time.
Observing that a Poincar\'e-Friedrichs/spectral gap estimate of the form $\norm{u}{L_2(\Omega)}\le C\norm{\nabla u}{L_2(\Omega)}$ is not applicable, and since an estimate of the type $\norm{u}{L_2(\Omega)}\le C\norm{u_x}{L_2(\Omega)}$ may \emph{not} hold, standard arguments involving Gr\"onwall's Lemma give
\begin{equation}\label{eq:basic_energy_estimate}
	\norm{u(t_f)}{L_2(\Omega)}^2\le e^{{\epsilon} t_f} \big({\epsilon}^{-1}\norm{f}{L_2(0,t_{{f}};L_2(\Omega))}^2+\norm{u_0}{L_2(\Omega)}^2\big),
\end{equation}
for any ${\epsilon}>0$, where we have used standard notation for Bochner spaces.
The presence of the term $e^{{\epsilon} t_f}$ mean that the stability estimate \eqref{eq:basic_energy_estimate} is not robust with respect to the final time.


The hypocoercivity of \eqref{eq:kolmogorov}, however, can be manifested through non-standard energy arguments.
More specifically, by considering different inner products  of form $(\cdot,\cdot)_{\rm hc}=(\cdot,\cdot)_{L_2(\Omega)}+(A\nabla\cdot,\nabla \cdot)_{L_2(\Omega)}$, where $A\in \mathbb{R}^{2\times2}_{\rm symm}$ is a suitable positive semidefinite matrix, and the associated norm $\norm{\cdot}{\rm hc}:=\sqrt{(\cdot,\cdot)}_{\rm hc}$, one arrives  at an energy estimate of the form
\begin{equation*}
	\frac{1}{2}\frac{\ud}{\ud t}\norm{u(t)}{\rm hc}^2+c\big(\norm{\nabla u}{L_2(\Omega)}^2+\norm{u_x}{\rm hc}^2\big)\le (f,u)_{\rm hc},
\end{equation*}
for some constant $c>0$.
Crucially, if $u$ is such that it satisfies a Poincar\'e-Friedrichs inequality, (e.g., satisfying suitable boundary conditions,) we can deduce
\begin{equation*}
	\frac{1}{2}\frac{\ud}{\ud t}\norm{u(t)}{\rm hc}^2+c_{\rm hc}\norm{u}{\rm hc}^2\le (f,u)_{\rm hc},
\end{equation*}
for some constant $c_{\rm hc}>0$ which depends on $A$, so that
$
	\frac{\ud}{\ud t}\norm{u(t)}{\rm hc}^2+c_{\rm hc}\norm{u}{\rm hc}^2\le c_{\rm hc}^{-1}\norm{f}{\rm hc}^2,
$
using standard arguments.
Thus, Gr\"onwall's Lemma gives
\begin{equation}\label{eq:energy_mod4}
	\norm{u(t_f)}{\rm hc}^2
	\le
	e^{-c_{\rm hc}t_f}\norm{u_0}{\rm hc}^2
	+
	c_{\rm hc}^{-1}\int_0^{t_f} e^{-c_{\rm hc}(t_f-s)} \norm{f(s) }{\rm hc}^2 \ud s.
\end{equation}
This is reminiscent of the standard parabolic decay estimates.
In particular, for $f = 0$, this provides decay as $t\to\infty$.

Retaining hypocoercive structures upon discretisation of respective PDE problems is desirable. This is because numerical methods admitting such structures would give rise to robust performance for long time computations. So far, only few works have been concerned with such \emph{numerical hypocoercivity} properties. We mention \cite{porretta_zuazua} for a study of hypocoercivity for a central difference method discretising the Kolmogorov equation \eqref{eq:kolmogorov}, using the abstract framework from \cite[Theorem 18]{villani}, and \cite{dujardin_herau_lafitte} where the approach from \cite{porretta_zuazua} was taken further to discuss the availability of spectral gap for central difference discretisations of the respective Fokker-Planck equation. Also, Foster, Loh\'eac \& Tran \cite{foster_loheac_tran} discuss the development of a Lagrangian-type splitting method based on a carefully constructed similarity transformation and linear finite elements over quasiuniform meshes, although no proof of decay to equilibrium is given for the numerical method itself. Also, Bessemoulin-Chatard \& Filbet \cite{besse_filbet} present design principles for the construction of equilibrium-preserving finite volume methods for nonlinear degenerate parabolic problems. Recently, Bessemoulin-Chatard, Herda \& Rey \cite{besse_FV} presented an asymptotic-preserving in the diffusive limit  finite volume scheme for a class of one-dimensional kinetic equations, using the hypocoercivity framework of Dolbeault, Mouhot \& Schmeiser \cite{DMS_TAMS}.

Developing Galerkin finite element methods which preserve/admit hypocoercivity structures is a largely unexplored area. This is despite the potential advantages of Galerkin methods, allowing for high order discretisations, over unstructured, possibly locally adapted meshes. We are only aware of \cite{georgoulis}, whereby a hypocoercivity-preserving family of spatially discrete non-conforming finite element methods for  \eqref{eq:kolmogorov} is constructed and analysed.  The key idea in the design of the methods in \cite{georgoulis} is the  construction of certain consistent stabilisations based on carefully selected numerical fluxes on the skeleton of the mesh, along with imposition of additional boundary conditions in the PDE problem. \rev{The modified inner product $(\cdot,\cdot)_{\rm hc}$, containing derivatives of order two or higher, results into the family of methods introduced in \cite{georgoulis} to scale} like a fourth order discrete operator in space. This may result to higher computational cost if used in conjunction with a fully discrete implicit time-stepping fashion and iterative linear solvers for each time-step.

This work is concerned with constructing a new, hypocoercivity-exploiting, stabilised finite element method, which address three key issues that arose in \cite{georgoulis}: 1) the new method scales like a second order discrete operator, thereby reinstating the expected computational complexity of discretising a second order differential operator, 2) both spatially discrete and fully discrete methods are presented, with the time discretisation realised by an arbitrary order discontinuous Galerkin timestepping approach, and 3) no additional boundary conditions other than standard Cauchy/Diri\-chlet-Neumann ones are required for the analysis of the new method.
On the other hand, however, the three aforementioned developments are possible \rev{at the cost of a discretization-parameter dependent,} negative Gr\"onwall exponent \rev{when} compared to the one in \cite{georgoulis}. In particular, the available spectral gap degenerates when increasing the dimension of the space-time approximation spaces (i.e., reducing meshsize/timestep and/or increasing polynomial degree). Nevertheless, for fixed discretisation parameters, hypocoercivity is proven and used, in turn, for the proof of \emph{a priori} error bounds which are robust with respect to the final time $t_f$.



The remainder of this work is structured as follows.
In Section \ref{sec:weak_form}, we discuss the special weak formulation of \eqref{eq:kolmogorov} which utilizes the concept of hypocoercivity in order to generate a contractive semigroup.
Following this, we discuss two semi-discretisations.
The first semi-discretisation is discussed in Section \ref{sec:SUPG} comes in the form of a streamline-upwinded Petrov-Galerkin method which does not itself immediately achieve an optimal rate of convergence.
This is followed in Section \ref{sec:fem} by a discretisation which makes use of both the properties discussed in Section \ref{sec:weak_form} and the streamline-upwinded Petrov-Galerkin method. Both of these semi-discretisations make use of continuous finite elements in space. Once the semi-discrete case has been treated, we move to Section \ref{sec:fullDiscrete} which considers a full discretisation of \eqref{eq:kolmogorov} by use of discontinuous Galerkin methods in time.


\subsection*{Notation}
To simplify notation, we abbreviate the $L_2(\omega)$-inner product and $L_2(\omega)$-norm for a Lebesgue-measurable subset $\omega\subset \mathbb{R}^d$ as $(\cdot,\cdot)_{\omega}$ and $\norm{\cdot}{\omega}$, respectively.
Moreover, when $\omega=\Omega$, we will further compress the notation to $(\cdot,\cdot)\equiv (\cdot,\cdot)_{\Omega}$ and $\norm{\cdot}{}\equiv \norm{\cdot}{\Omega}$.
The standard notation $H^k(\omega)$ for Hilbertian Sobolev spaces, $k\in\mathbb{R}$ will be used.
In addition, given an interval $J\subset \mathbb{R}$ and a Banach space $V$, we use the standard notation Bochner spaces that are denoted by $L_p(J; V)$.

For the rest of this work, we restrict to $\Omega$ being a polygonal domain and introduce a family of triangulations of $\Omega$, say $\mathcal{T}$, consisting of mutually disjoint open triangular elements $T\in \mathcal{T}$, whose closures cover $\bar{\Omega}$ exactly.
Also, let $h\colon\cup_{T\in\mathcal{T}} T\to\mathbb{R}_+$ be the local meshsize function defined elementwise by $h|_T:=h_T:={\rm diam}(T)$.
For simplicity, we further assume that $\mathcal{T}$ is shape-regular, in the sense that the radius $\rho_T$ of the largest inscribed circle of each $T\in\mathcal{T}$ is bounded from below with respect to each element's diameter $h|_T$, uniformly as $\|h\|_{L_\infty(\Omega)}\to 0$ under mesh refinement.
Also, we assume that $\mathcal{T}$ is locally quasi-uniform in the sense that the diameters of adjacent elements are uniformly bounded from above and below.
Finally, let $\Gamma:=\cup_{T\in\mathcal{T}}\partial T$ denote the mesh skeleton, and $\Gamma_{\rm int}:=\Gamma\backslash\partial\Omega$.

The family of finite element spaces subordinate to $\mathcal{T}$ we consider are defined by
\[
	V_h\equiv V_h^p:=\{V\in H^1_{-}(\Omega): V|_T\in \mathbb{P}_p(T),\ T\in\mathcal{T}\},
\]
with $\mathbb{P}_p(\omega)$, $\omega\subset\mathbb{R}^d$, denoting the space of polynomials of total degree at most $p$ over $\omega$, $p=1,2,3,\ldots$. Notice that this choice of $V_h$ is only $H^1$-conforming despite the higher order derivatives which \rev{will later} appear in the \rev{bilinear form of interest}.
%
 Further, we define the \emph{broken Sobolev spaces} $
H^r(\Omega,\mathcal{T}):=\{v|_T\in H^r(T), T\in\mathcal{T}\},
$ and the broken gradient $\nabla_\mathcal{T}$ to be the element-wise gradient operator with $\nabla_{\mathcal{T}}w|_T=\nabla w|_T$, $T\in\mathcal{T}$, for $w\in H^1(\Omega,\mathcal{T})$.

\section{A special weak formulation}\label{sec:weak_form}

The problem of closing a degenerate parabolic PDE with suitable boundary conditions is well understood via the classical theory of linear second order equations with non-negative characteristic form \cite{fichera,or73}.
In particular, with $\mathbf{n}(\cdot):=(n_1(\cdot),n_2(\cdot))^T$ denoting the unit outward normal vector at almost every point of $\partial\Omega$, which is assumed to be piecewise smooth and Lipschitz, we first define the elliptic portion of the boundary
\[
	\partial_0\Omega:=\{(x,y)\in \partial \Omega: n_1(x,y)\ne 0\}.
\]
On the non-elliptic portion of the boundary $\partial \Omega\backslash\partial \Omega_0$, we define the inflow and outflow boundaries:
\[
	\partial_-\Omega:= \{(x,y) \in \partial\Omega\backslash\partial_0\Omega: xn_2(x,y )<0\}, \quad \partial_+\Omega:= \{\xi \in  \partial\Omega\backslash\partial_0\Omega: xn_2(x,y)   \ge  0\}.
\]
For notational brevity, any characteristic portions of the boundary $\{(x,y) \in  \partial\Omega\backslash\partial_0\Omega: xn_2(x,y)  = 0\}$ have been included into the outflow part, since their treatment is identical in what follows.
Notice that $\partial \Omega = \partial_+ \Omega \cup \partial_- \Omega \cup \partial_0 \Omega$.
Introducing the notation $ Lu\equiv -u_{xx}+xu_y$, we consider the initial/boundary-value problem:
	\begin{equation}\label{eq:kolmogorov-bounded}
	\begin{aligned}
u_t+Lu= u_t -u_{xx}+xu_y =&\ f, &\text{in } (0,t_f]\times\Omega,\\
u=&\ u_0, &\text{on } \{0\}\times\Omega,\\
u=&\ 0, &\text{on } (0,t_f]\times\partial_{-}\Omega,\\
n_1 u_x=&\ 0, &\text{on } (0,t_f]\times \partial_0 \Omega,
\end{aligned}
\end{equation}
for $t_f>0$ and for $f\in L^2((0,t_f); H^1(\Omega))$,  noting the {stronger,} non-standard regularity assumption for $f$.
The well-posedness of the above problem is assured upon assuming that $\partial_{-}\Omega$ has positive one-dimensional Hausdorff measure \cite{or73}. \rev{We refer to \cite{or73} for further details on the admissible boundary conditions for ultraparabolic PDEs and their physical meaning. We  note that the developments below (with minor modifications) still apply upon replacing the non-flux condition by a Dirichlet one. We opt for the no-flux condition in the present work for its relevance in kinetic-type systems.}
For brevity, denote the time interval by $I:=(0,t_f]$. We note that, in contrast to the previous work \cite{georgoulis}, in what follows, we do \emph{not} prescribe any additional boundary conditions.

Let $A$ be a symmetric and non-negative definite matrix of the form
\begin{equation}\label{eq:AForm}
	A :=\begin{pmatrix}\alpha & \beta\\ \beta & \gamma \end{pmatrix}
\end{equation}
where $\alpha,\beta,\gamma$ are non-negative parameters.
Denote by $\Ainner{\cdot}{\cdot}$ the inner product
\[
	\Ainner{w}{v}:=(w,v)+ (\nabla w,A\nabla v ),
\]
for $w,v\in H^1(\Omega)$ with induced inner product $\norm{\cdot}{A}:= \Ainner{\cdot}{\cdot}^{\frac{1}{2}}$.
This norm will play the role of $\norm{\cdot}{hc}$ from the Introduction.

Assuming sufficient regularity for the exact solution $u$, so that the following calculations are well defined, \eqref{eq:kolmogorov} implies
\begin{equation}\label{eq:weak_hypo}
\Ainner{u_t}{v}+ \rev{\Ainner{L u}{v}}= \Ainner{f}{v},
\end{equation}
for all $v\in H^1_{-}(\Omega):=\{v\in H^1(\Omega):v|_{\partial_{-}\Omega}=0\}$, and set
\begin{equation}\label{eq:spatial_part}
a(u,v):=(u_x,v_x)+(xu_y,v)+(\nabla L u, A\nabla v).
\end{equation}

\subsection{The hypocoercive structure of \texorpdfstring{$a(\cdot,\cdot)$}{a}}\label{sec:hypo}
Setting $v=u$ into \eqref{eq:spatial_part} gives
\[
a(u,u)= \|u_x\|^2+(xu_y,u)+(-u_{xxx}+(xu_y)_x,\alpha u_x+\beta u_y)
+ (-u_{xxy}+xu_{yy},\beta u_x+\gamma u_y),
\]
which, upon integration by parts, gives
\begin{equation}\label{eq:hypoco3}
	\begin{aligned}
		a(u,u)=&\ \norm{u_x}{}^2+\frac{1}{2}\norm{\sqrt{xn_2}u}{\partial_+\Omega}^2+\norm{\sqrt{A}\nabla u_x}{}^2+\beta\norm{u_y}{}^2+\alpha(u_y, u_x)\\
		&+(\frac{1}{2}xn_2\nabla u-n_1\nabla u_{x}, A\nabla u)_{\partial\Omega}.
	\end{aligned}
\end{equation}
 In \cite{georgoulis} additional boundary conditions are considered, so that the last two terms on the right-hand side of \eqref{eq:hypoco3} are non-negative and, thus, they can be omitted.
{After applying Young's and Cauchy-Schwarz' inequalities} and electing $\alpha>0$, $\beta=4\alpha^2/9$, and $\gamma=\alpha^3/3$, for instance, we deduce
\begin{equation}\label{eq:hypoco}
a(u,u)\ge \frac{1}{4}\|u_x\|^2+\frac{\alpha}{3} \|u_{xx}\|^2+\frac{\alpha^2}{9}\|u_y\|^2+\frac{\alpha ^3}{27}\|u_{xy}\|^2.
\end{equation}
Thus, remarkably, using the non-standard test function $v-\nabla\cdot (A\nabla v)$ for well-chosen $A$, results in $L$ providing coercivity with respect to $u_y$ also: this is a manifestation of the concept of hypocoercivity, as $\partial_y=[x\partial_y,\partial_x]$; we refer to \cite[Theorem 24]{villani} for details.

\begin{remark}
	To emphasise the potential generality of the approach, one may consider a $(d+1)$-dimensional version of \eqref{eq:kolmogorov}: for $\Omega\subset\mathbb{R}^{d+1}$ open, find $u\colon (0,t_f]\times \Omega\to\mathbb{R}$, such that
	\begin{equation}\label{eq:kolmogorov-bounded_dD}
		\begin{aligned}
			u_t+L_{d+1}u=u_t - u_{x_1x_1}+  \sum_{i=1}^d x_i u_{x_{i+1}} =&\ f, &\text{in } (0,t_f]\times\Omega,
		\end{aligned}
	\end{equation}
	for $t_f>0$ and $f\in H^1(\Omega)$, along with suitable initial/boundary conditions.
	Interestingly, \eqref{eq:kolmogorov-bounded_dD} is smoothing, despite possessing explicit diffusion in \emph{one spatial dimension only}. Indeed, we have, respectively,
	$
	[\partial_{x_i},\mbf{b}\cdot\nabla]=\partial_{x_{i+1}}
	$
	thereby, H\"ormander's rank condition is satisfied \cite{hormander}, implying that \eqref{eq:kolmogorov-bounded_dD} is, in fact, hypoelliptic! Moreover, since full rank is achieved via commutators involving the skew-symmetric 1st order part  $\mbf{b}\cdot\nabla u$, the PDE in \eqref{eq:kolmogorov-bounded_dD} satisfies also the commutator hypotheses of \cite[Theorem 24]{villani}. The developments discussed in this work can be transferred to this problem also with minor modifications only.
\end{remark}

\section{A streamline-upwinded Petrov-Galerkin formulation}\label{sec:SUPG}

For equations with non-negative characteristic form, such as  \eqref{eq:kolmogorov-bounded}, a streamline-upwinded Petrov-Galerkin formulation (SUPG) was proposed \cite{houston_suli_stabilised_nonneg} to achieve optimal rates of convergence in the presence of degenerate diffusion for continuous finite element spaces like $V_h^p$.  In the present setting, the method in \cite{houston_suli_stabilised_nonneg} reads: find $U\equiv U(t)\in V_h$, $t\in(0,t_f]$, such that
\begin{equation}\label{eq:FEM_SUPG}
	(U_t,V) + a_{h,{\rm SUPG}}(U,V) = (f, V+ \tau (V_t+ xV_y)),
\end{equation}
for all  $V\in V_h$,
with
\begin{equation}\label{eq:SUPG_term}
	a_{h,{\rm SUPG}}(U,V):=(U_x,V_x)+ (xU_y,V)+\sum_{T\in\mathcal{T}}(U_t+LU,\tau (V_t+ xV_y))_T,
\end{equation}
for an SUPG parameter $\tau\colon \Omega\to \mathbb{R}$ defined element-wise precisely below.
We also set $U(0):=\pi u_0\in V_h$, with $\pi\colon L_2(\Omega)\to V_h$ denoting the orthogonal $L_2$-projection onto the finite element space. \rev{The intuition behind testing the PDE against the non-standard  term $\tau (V_t+ xV_y)$ is to ``symmetrize'' the bilinear form with respect to the first order terms.}

\subsection{\rev{Coercivity in SUPG-norm}}

Setting $U=V=w\in V_h^p$ in \eqref{eq:SUPG_term},  we have the coercivity relation:
\begin{equation}\label{coer_SUPG}
	\begin{aligned}
		a_{h,{\rm SUPG}}(w,w)
		=&\ \norm{w_x}{}^2 + \frac{1}{2}\norm{\sqrt{xn_2}w}{\partial_+\Omega}^2+\norm{\sqrt{\tau} (w_t + x w_y)}{}^2 \\
&-\sum_{T\in\mathcal{T}}(w_{xx},\tau (w_t+xw_y))_T\\
		\ge &\ \norm{w_x}{}^2 + \frac{1}{2}\norm{\sqrt{xn_2}w}{\partial_+\Omega}^2 +\norm{\sqrt{\tau} (w_t + x w_y)}{}^2\\
		&- \norm{C_{INV}^{}p^2h^{-1}\sqrt{\tau}w_{x}}{}\norm{\sqrt{\tau} (w_t+ xw_y)}{}
		\ge  \frac{1}{2}\ndg{w}_{\rm SUPG}^2,
	\end{aligned}
\end{equation}
for
\[
\ndg{w}_{\rm SUPG}= \Big(\norm{w_x}{}^2 + \norm{\sqrt{xn_2}w}{\partial_+\Omega}^2 +\norm{\sqrt{\tau} (w_t + x w_y)}{}^2\Big)^{\frac{1}{2}},
\]
whereby, for each $T\in\mathcal{T}$,  we set
\begin{equation}\label{def:tau semi-dsicrete}
	\begin{aligned}
 \tau|_T := C_{INV}^{-2} p^{-4}h^{2},
	\end{aligned}
\end{equation}
with  $C_{INV}^{}$ the constant of the inverse estimate $\norm{\nabla v}{T}\le C_{INV}p^2h_T^{-1}\norm{v}{T}$, valid for $v\in \mathbb{P}_p(T)$, depending only on the shape-regularity of the mesh.

For $f=0$, starting from \eqref{eq:FEM_SUPG}, we set $V=U$ and, employing \eqref{coer_SUPG} along with standard arguments, we arrive at
\[
	\frac{\ud }{\ud t}\|U\|^2 +\ndg{U}_{\rm SUPG}^2	\le 0,
\]
which implies the expected stability:
\begin{equation}\label{SUPG_stab}
	\|U(t_f)\|\le \|U(0)\|,
\end{equation}
for all $\tau\ge 0$, i.e., we have stability of the SUPG scheme with respect to the PDE initial data.

\subsection{\rev{Error estimate}}
Set now
\[
	\rho:=u-\pi u,\qquad \vartheta:=\pi u -U,
\]
so that $e=\rho+\vartheta$, whereby, following \cite{chrysafinos_hou},  $\pi\colon L_2(\Omega)\to V_h^p$ is the orthogonal $L_2$-projection operator onto the finite element space, giving $(\rho_t,\vartheta)=0$.

Employing \eqref{eq:FEM_SUPG} with $V=\vartheta$ gives the error equation
\begin{equation}\label{eq:error_eq1}
	\begin{aligned}
	&	\frac{1}{2}\frac{\ud }{\ud t}\|\vartheta\|^2+a_{h,{\rm SUPG}}(\vartheta,\vartheta)\\
		=&\ 	((\pi u)_t,\vartheta)+a_{h,{\rm SUPG}}(\pi u,\vartheta) -(f, \vartheta_t+ \tau (\vartheta_t+ x\vartheta_y)),\\
		=&\ 	-(\rho_t,\vartheta)-a_{h,{\rm SUPG}}(\rho,\vartheta)=-a_{h,{\rm SUPG}}(\rho,\vartheta).
	\end{aligned}
\end{equation}
Using $(\rho,\vartheta_t)=0$, we have $(\rho,x\vartheta_y)=(\rho,x\vartheta_y + \vartheta_t)$. Then, we have,
\begin{equation}\label{SUPG:cont}
	\begin{aligned}
		a_{h,{\rm SUPG}}(\rho,\vartheta)
		=&\ (\rho_x,\vartheta_x)-(\rho,x\vartheta_y)+\int_{\partial_+\Omega}xn_2\rho\vartheta\ud s \\&	+(\rho_t + x\rho_y,\tau ( \vartheta_t+x\vartheta_y))-\sum_{T\in\mathcal{T}}(\rho_{xx},\tau (\vartheta_t+ x\vartheta_y))_T\\
		\le&\ \Big( \norm{\rho_x}{}+\norm{\tau^{-\frac{1}{2}}\rho}{}+\norm{\sqrt{xn_2}\rho}{\partial_+\Omega}+ \norm{\sqrt{\tau}(\rho_t + x\rho_y)}{}\\   \qquad&\qquad+\big(\sum_{T\in\mathcal{T}} \norm{\sqrt{\tau}\rho_{xx}}{T}^2\big)^{\frac{1}{2}} \Big)\\
		&\qquad \times \Big( \norm{\vartheta_x}{}+\norm{\sqrt{xn_2}\vartheta}{\partial_+\Omega}+2\norm{\sqrt{\tau}(\vartheta_t + x\vartheta_y)}{}  \Big)\\
		\le &\	4 \Big(\ndg{\rho}_{{\rm SUPG}}+\norm{\tau^{-\frac{1}{2}}\rho}{}+(\sum_{T\in\mathcal{T}} \norm{\sqrt{\tau}\rho_{xx}}{T}^2)^{\frac{1}{2}} \Big)\ndg{\vartheta}_{{\rm SUPG}},
	\end{aligned}
\end{equation}
\rev{where $\tau$ is defined in \eqref{def:tau semi-dsicrete}.} Using \eqref{SUPG:cont} along with \eqref{coer_SUPG} into \eqref{eq:error_eq1}, we arrive at
\begin{equation}\label{SUPG_error_est}
	\begin{aligned}
		\frac{\ud }{\ud t}\|\vartheta\|^2 +\frac{1}{2}\ndg{\vartheta}_{{\rm SUPG}}^2
		\le &\ 	32 \Big(\ndg{\rho}_{{\rm SUPG}}+\norm{\tau^{-\frac{1}{2}}\rho}{}+(\sum_{T\in\mathcal{T}} \norm{\sqrt{\tau}\rho_{xx}}{T}^2)^{\frac{1}{2}} \Big)^2.
	\end{aligned}
\end{equation}
We note that $\ndg{\cdot}_{{\rm SUPG}}$ may only be a seminorm, in general, in $V_h^p$, for instance when $(0,y)\in \Omega$ for some $y\in\mathbb{R}$. In particular, it is \emph{not} necessarily true that $\norm{\vartheta}{}\le C_{PF}\ndg{\vartheta}_{{\rm SUPG}}$, for some constant $C_{PF}$, independent of $\vartheta$. Therefore, from \eqref{SUPG_error_est}, using standard arguments, we can only conclude the \emph{a priori} error estimate
\begin{equation}\label{SUPG_error_est_a_priori}
	\begin{aligned}
\norm{e(t_f)}{}^2
&+\int_0^{t_f}\ndg{e}_{{\rm SUPG}}^2\ud t
	\le\   \norm{e(0)}{}^2+		2\norm{\rho(t_f)}{}^2\\
	 &+128 \int_0^{t_f}\Big(\ndg{\rho}_{{\rm SUPG}}+\norm{\tau^{-\frac{1}{2}}\rho}{}+(\sum_{T\in\mathcal{T}} \norm{\sqrt{\tau}\rho_{xx}}{T}^2)^{\frac{1}{2}} \Big)^2\ud t .
\end{aligned}
\end{equation}
Assume, now, optimal best approximation properties for the projection $\pi$, viz.,
\begin{equation}\label{best_approx}
	\norm{h^{\kappa-1}\rho}{}^2+\norm{h^{\kappa}\nabla \rho}{}^2
+ \sum_{T\in\mathcal{T}}
(\norm{h^{\kappa+\frac{1}{2}} \nabla \rho}{\partial T}^2  + \norm{h^{\kappa+1} \nabla^2\rho}{T}^2 )
\le C_{app} \sum_{T\in\mathcal{T}}(|h^{\kappa+s}u|^2_{H^{s+1}(T)}),
\end{equation}
for $u\in H^1(I; H^{2}(\Omega) \cap H^{s+1}(\Omega,\mathcal{T}))$, $s\ge 1$, and $\kappa\in \mathbb{R}$, with $\nabla^2$ denoting the (weak) Hessian of its argument.
Assuming, now, sufficient space-time regularity for the exact solution $u$, the above best approximation estimates yield the error bound
\begin{equation}\label{SUPG_reg_error}
	\begin{aligned}
	\|e(t_f)\|^2+ \int_0^{t_f}\ndg{e}_{{\rm SUPG}}^2\ud t
	\le  &\ \norm{e(0)}{}^2
	+
	C\sum_{T\in\mathcal{T}}\Big( \max_{t\in I} |h^{s}u|^2_{H^{s}(T)}  \Big)\\
	&+
	C\sum_{T\in\mathcal{T}}\Big(\int_0^{t_f} \big( |h^{s}u|_{H^{s+1}(T)}^2+|h^{s+1}u_t|_{H^{s+1}(T)}^2\big)\ud t\Big),
\end{aligned}
\end{equation}
for some $C>0$, independent of $t_f$ and of $u$.
Thus, if the exact solution were to satisfy  $u\in H^1(I; H^{2}(\Omega) \cap H^{s+1}(\Omega,\mathcal{T})) $ for  $1 \leq s\leq p$,    \eqref{SUPG_reg_error} implies
\begin{equation}\label{SUPG_reg_error_final}
	\begin{aligned}
		\|e(t_f)\|^2 +	\int_0^{t_f}\ndg{e}_{{\rm SUPG}}^2\ud t
		\le  &\	 C t_f h_{\max}^{2s},
	\end{aligned}
\end{equation}
with $h_{\max}=\max_{x\in\Omega} h$ for some $C>0$ constant depending on $u$.
\begin{remark}[\rev{$\tau\equiv 0$}]
We remark that,  for $\tau\equiv 0$ in \eqref{eq:FEM_SUPG} and employing energy estimates, the resulting `plain' Galerkin semi-discrete method may admit stronger dependence on $t_f$. Indeed, setting $\tau=0$ and working as above,  \eqref{coer_SUPG} degenerates to:
\begin{equation}\label{coer_SUPG_t0}
	\begin{aligned}
		a_{h,{\rm SUPG}}(\vartheta,\vartheta)
		=&\ \norm{\vartheta_x}{}^2 + \frac{1}{2}\norm{\sqrt{xn_2}\vartheta}{\partial_+\Omega}^2,
	\end{aligned}
\end{equation}
while \eqref{SUPG:cont} gives
\begin{equation*}
	\begin{aligned}
		a_{h,{\rm SUPG}}&(\rho,\vartheta)
		=\ (\rho_x,\vartheta_x)-(\rho,x\vartheta_y)+\int_{\partial_+\Omega}xn_2\rho\vartheta\ud s\\
\le&\ \Big( \norm{\rho_x}{}+\norm{C_{INV}^{}p^2h^{-1}x\rho}{}+\norm{\sqrt{xn_2}\rho}{\partial_+\Omega}\Big)\Big( \norm{\vartheta_x}{}+\norm{\sqrt{xn_2}\vartheta}{\partial_+\Omega}+ \norm{\vartheta}{}\Big),
	\end{aligned}
\end{equation*}
so that, \eqref{eq:error_eq1} and standard arguments give
\[
\begin{aligned}
	\frac{1}{2}\frac{\ud }{\ud t}\|\vartheta\|^2&+\frac{1}{4}\norm{\vartheta_x}{}^2 + \frac{1}{4}\norm{\sqrt{xn_2}\vartheta}{\partial_+\Omega}^2 \\
	\le &\ 2\Big( \norm{\rho_x}{}+\norm{C_{INV}^{}p^2h^{-1}x\rho}{}+\norm{\sqrt{xn_2}\rho}{\partial_+\Omega}\Big) \norm{\vartheta}{},
\end{aligned}
\]
or
\[
\begin{aligned}
	\frac{\ud }{\ud t}\|\vartheta\|^2
	\le &\ \frac{16}{\epsilon}\Big( \norm{\rho_x}{}+\norm{C_{INV}^{}p^2h^{-1}x\rho}{}+\norm{\sqrt{xn_2}\rho}{\partial_+\Omega}\Big)^2+\epsilon \norm{\vartheta}{}^2,
\end{aligned}
\]
for any $\epsilon>0$.
The triangle inequality, together with \eqref{best_approx}, then, implies
\begin{equation}\label{result_old}
\begin{aligned}
	\|e(t_f)\|^2
	\le &\ C(\epsilon^{-1})e^{\epsilon t_f} h_{\max}^s,
\end{aligned}
\end{equation}
for some constant $C>0$ depending proportionally on $\epsilon^{-1}$ and on $u$, assuming the same regularity setting as in \eqref{SUPG_reg_error_final}.
\end{remark}

\section{A hypocoercivity-exploiting SUPG method}\label{sec:fem}

Starting from \eqref{eq:weak_hypo}, we consider the spatially discrete finite element method: find $U\equiv U(t)\in V_h$, $t\in(0,t_f]$, such that
\begin{equation}\label{eq:FEM}
\Ainner{U_t}{V}
+a_h(U,V)
=\Ainner{f}{V}+(f, \tau (V_t+ xV_y)),
\end{equation}
for all  $V\in V_h$,  for\begin{equation}\label{eq:spatial_part_h}
	a_h(U,V):=a_{h,{\rm SUPG}}(U,V)+\sum_{T\in\mathcal{T}}(\nabla L U, A\nabla V)_T.
\end{equation}
We set $U(0):=\hat{\pi} u_0\in V_h$, with $\hat{\pi}\colon H^1_-(\Omega)\to V_h$ defined as the $A$-orthogonal projection operator onto the finite element space $V_h$. 
\begin{remark}
	The weak form \eqref{eq:FEM} is a Petrov-Galerkin formulation arising formally by testing the PDE against $v + \tau (v_t + xv_y)-\nabla\cdot (A\nabla v)$, modulo a boundary term arising from integration by parts.
\end{remark}

\subsection{\rev{Numerical hypocoercivity}}

\rev{To study the numerical hypocoercivity, we first define}
\[
\begin{aligned}
	\ndg{w}:= \Big(&\ \frac{1}{2}\norm{w_x}{}^2+\norm{\sqrt{\gamma\delta}w_y}{}^2+\frac{1}{2}\norm{\sqrt{\tau} (w_t+xw_y)}{}^2+\norm{\sqrt{A}\nabla_{\mathcal{T}} w_x}{}^2\\
	&+\norm{\sqrt{xn_2}w}{\partial_+\Omega}^2+ \sum_{T\in\mathcal{T}}\norm{\sqrt{xn_2A}\nabla w}{\partial_+ T}^2\Big)^{\frac 1 2},
\end{aligned}
\]
with $A$ as in \eqref{eq:AForm} \rev{and $\nabla_\mathcal{T}$ is the element-wise gradient operator, defined in in the last section of the Introduction.}
We have the following coercivity result.

\begin{lemma}[(Hypo)coercivity]\label{lem:coercivity} 	For all $w\in V_h$, we have
	\begin{equation}\label{lemma:hypocoer}
		a_h(w,w)\ge \frac {1}{4} \ndg{w}^2,
	\end{equation}
 when $\alpha = (8\delta)^{-1}$, $\beta =(24\delta^2)^{-1}$, $\gamma=(64\delta^3)^{-1}$, whereby for each $T\in\mathcal{T}$,
	\begin{equation}\label{def: delta and tau}
\delta|_T:=	\delta_T=  \rev{ \max\left\{\frac{C_{inv}^2p^2\norm{xn_2}{L_\infty(\partial_-T)}}{h_T},\frac{2\norm{n_1}{L_\infty(\partial T)}^2 C_{inv}^4p^4}{3h_T^2}\right\}}
 \text{and }
 \tau|_T :=  \frac{h_T^{2}}{4C_{INV}^{2}p^{4}} ,
	\end{equation}
	where
$C_{inv}>0$ the trace-inverse inequality constant depending only on the shape-regularity of the mesh, viz., $\norm{v}{\partial T}\le C_{inv}ph_T^{-\frac{1}{2}}\norm{v}{T}$, valid for $v\in \mathbb{P}_p(T)$, $T\in\mathcal{T}$.

\end{lemma}
\begin{proof}
	Now, working as above, we have
	\[
	\begin{aligned}
			a_h(w,w)=&\ \norm{w_x}{}^2+\frac{1}{2}\norm{\sqrt{xn_2}w}{\partial_+\Omega}^2+\norm{\sqrt{\tau} (w_t + xw_y)}{}^2\\ &+\sum_{T\in\mathcal{T}}\Big( (\nabla L w, A\nabla w)_T-(w_{xx},\tau (w_t+xw_y))_T\Big),
	\end{aligned}
	\]
whilst the last term on the right-hand side gives
		\[
	\begin{aligned}
\mathcal{I}:=&\ \norm{\sqrt{A}\nabla w_x}{T}^2+\beta\norm{w_y}{T}^2+\alpha(w_y, w_x)_T
	+(\frac{1}{2}xn_2\nabla w-n_1\nabla w_{x}, A\nabla w)_{\partial T}\\
	&-(w_{xx},\tau (w_t+xw_y))_T\\
	\ge &\norm{\sqrt{A}\nabla w_x}{T}^2+\beta\norm{w_y}{T}^2-\frac{\alpha^2}{2\epsilon_1}\norm{w_y}{T}^2- \frac{\epsilon_1}{2}\norm{w_x}{T}^2
	-\frac{1}{2}\norm{\sqrt{|xn_2|A}\nabla w}{\partial_- T}^2\\
&+\frac{1}{2}\norm{\sqrt{xn_2A}\nabla w}{\partial_+ T}^2-\norm{n_1}{L^\infty(\partial T)}\norm{\sqrt{A}\nabla w_{x}}{\partial T} \norm{\sqrt{A}\nabla w}{\partial T}\\
& -\frac{1}{2}\norm{C_{INV}^{}p^2h^{-1}\sqrt{\tau}w_{x}}{T}^2-\frac{1}{2}\norm{\sqrt{\tau} (w_t+ xw_y)}{T}^2\\
	\ge &\ \Big(1-\epsilon_2\Big)\norm{\sqrt{A}\nabla w_x}{T}^2+\Big(\beta-\frac{\alpha^2}{2\epsilon_1}\Big)\norm{w_y}{T}^2- \Big(\frac{\epsilon_1}{2}+\frac{1}{8}\Big)\norm{w_x}{T}^2\\
&	-\Big(\rev{\frac{C_{inv}^2p^2\norm{xn_2}{L_\infty(\partial_-T)}}{2h_T}+\frac{\norm{n_1}{L_\infty(\partial T)}^2 C_{inv}^4p^4}{4h_T^2\epsilon_2}}\Big)\norm{\sqrt{A}\nabla w}{ T}^2\\
&
-\frac{1}{2}\norm{\sqrt{\tau} (w_t + xw_y)}{T}^2+\frac{1}{2}\norm{\sqrt{xn_2A}\nabla w}{\partial_+ T}^2,
\end{aligned}
\]
upon selecting $\tau:=C_{INV}^{-2}p^{-4}h^{2}/4$.
We set $\epsilon_2 = \frac{3}{4}$, and recall the definition of $\delta_T$, for which we note that $\delta_T\neq 0$ since $n_1^2+n_2^2=1$ pointwise.
Continuing, we have
\[
\begin{aligned}
\mathcal{I}\ge &\ \frac{1}{4}\norm{\sqrt{A}\nabla w_x}{T}^2+\Big(\beta-\frac{\alpha^2}{2\epsilon_1}\Big)\norm{w_y}{T}^2-\Big(\frac{\epsilon_1}{2}+\frac{1}{8}\Big)\norm{w_x}{T}^2
	-\delta_T\alpha \norm{w_x}{ T}^2\\
&-\delta_T\gamma \norm{w_y}{ T}^2-2\delta_T\beta(w_x,w_y)_{ T}+\frac{1}{2}\norm{\sqrt{xn_2A}\nabla w}{\partial_+ T}^2-\frac{1}{2}\norm{\sqrt{\tau} (w_t + xw_y)}{T}^2\\
\ge &\
\frac{1}{4}\norm{\sqrt{A}\nabla w_x}{T}^2+\Big(\beta-\frac{\alpha^2}{2\epsilon_1}-\delta_T\gamma- \delta_T\frac{\beta^2}{\alpha \epsilon_3}\Big)\norm{w_y}{T}^2\\
&-\Big( \frac{\epsilon_1}{2}+\frac{1}{8}+(1+\epsilon_3)\delta_T\alpha\Big)\norm{w_x}{T}^2
+\frac{1}{2}\norm{\sqrt{xn_2A}\nabla w}{\partial_+ T}^2-\frac{1}{2}\norm{\sqrt{\tau} (w_t + xw_y)}{T}^2.
	\end{aligned}
	\]
Setting $\epsilon_1=1$, $\epsilon_3 = 1$, $\alpha = (8\delta_T)^{-1}$,
$\beta = (24\delta_T^{2})^{-1}$, and $\gamma = \frac{9}{8}\frac{\beta^2}{\alpha} =  (64 \delta_T^3)^{-1}$, we deduce
	\[
\begin{aligned}
\mathcal{I}
	\ge&\ \frac{1}{4}\norm{\sqrt{A}\nabla w_x}{T}^2 + \frac{5\delta_T \gamma }{18}\norm{w_y}{T}^2- \frac{7}{8}\norm{w_x}{T}^2\\
	&+\frac{1}{2}\norm{\sqrt{xn_2A}\nabla w}{\partial_+ T}^2-\frac{1}{2}\norm{\sqrt{\tau} (w_t + xw_y)}{T}^2.
\end{aligned}
\]
By combining the above estimates, we deduce
	\[
\begin{aligned}
a_h(w,w)\ge&\ \frac{1}{8}\norm{w_x}{}^2+ \frac{5}{18}\norm{\sqrt{\gamma\delta}w_y}{}^2+\frac{1}{4}\norm{\sqrt{A}\nabla_{\mathcal{T}} w_x}{}^2+\frac{1}{2}\norm{\sqrt{xn_2}w}{\partial_+\Omega}^2\\
&+\frac{1}{2} \sum_{T\in\mathcal{T}}\norm{\sqrt{xn_2A}\nabla w}{\partial_+ T}^2 +\frac{1}{2}\norm{\sqrt{\tau} (w_t + xw_y)}{}^2,
\end{aligned}
\]
from which the result follows.
\end{proof}

\begin{remark}\label{remark_abc}
	The choice for $\alpha,\beta,\gamma$ is not unique in the above proof. The particular numbers have been chosen as a typical example of possible values.
\end{remark}

The following result provides a mesh-dependent spectral gap.
\begin{lemma}\label{hypoco_stab}
	Assume that $h_{\max}=\max_{x\in \Omega} h<1$, and set $h_{\min}:=\min_{x\in\Omega}h$. Then, there exists a constant $1>c_{\rm hc}>0$, independent of $h$ and $p$, such that, for all $w\in H^1(I; H_-^1(\Omega)\cap H^2(\Omega,\mathcal{T})) $, we have
	\[
	\ndg{w}^2\ge \kappa \norm{w}{A}^2,
\qquad
\text{with}
\qquad
\kappa\equiv\kappa(h,p):= c_{\rm hc}h_{\min}^4p^{-8}.
	\]
\end{lemma}
\begin{proof} Let $B:= {\rm diag}(1,(32\delta^2)^{-1})$. Then we have, respectively,
	\[
		2	\ndg{w}^2\ge 
		\norm{\sqrt{B}\nabla w}{}^2
		= \norm{\sqrt{B-\nu A}\nabla w}{}^2+\nu\norm{\sqrt{A}\nabla w}{}^2
		.
		\]
		Upon defining $\bar{\delta}:=\min_{T\in\mathcal{T}}C_\delta^Tp^4h_{\min}^{-2}$, we immediately have $\delta\ge \bar{\delta}$.
		Selecting now $\nu=\bar{\delta}/2$, the eigenvalues $\lambda_i$, $i=1,2$ of $B-\nu A$ satisfy $ \frac{1}{2}\le \lambda_1\le 1$ and $ \frac{h_{\min}^4}{384p^8}\le \lambda_2< \frac{1}{2}$,  which implies
\[
 \norm{\sqrt{B-\nu A}\nabla w}{}^2 \geq \lambda_2  \norm{\nabla w}{}^2 \geq C_{\rm PF}^{-1} \lambda_2  \norm{w}{}^2,
\]
where $C_{\rm PF}>0$ is the Poincar\'e-Friedrichs/spectral gap inequality $ \norm{w}{}^2 \leq C_{\rm PF}\norm{\nabla w}{}^2$ for any $w\in H^1_-(\Omega)\cap H^2(\Omega, \mathcal{T})$, independent of $w$ and discretisation parameters. Thus the result follows for $c_{\rm hc}=2^{-1}\min\{(192C_{\rm PF})^{-1}, \min_{T\in\mathcal{T}}C_\delta^T\}$.
	\end{proof}

On the other hand, for the hypocoercivity-exploiting SUPG method \eqref{eq:FEM}, for $f = 0 $ and $V=U$ we have, respectively,
\[
\begin{aligned}
	\frac{1}{2}\frac{\ud}{\ud t}\norm{U}{A}^2+ \frac{\kappa}{4} \norm{U}{A}^2\le 		\frac{1}{2}\frac{\ud}{\ud t}\norm{U}{A}^2+	\frac{1}{4}\ndg{U}^2
\le \frac{1}{2}\frac{\ud}{\ud t}\norm{U}{A}^2+a_h(U,U)= 0,
\end{aligned}
\]
implying
\begin{equation}\label{hypoSUPG_stab}
	\|U(t_f)\|\le e^{-\kappa t_f/4}\|U(0)\|,
\end{equation}
which improves upon \eqref{SUPG_stab}. Therefore, the ``hypocoercivity-inducing'' stabilisation results into favourable dependence of the error with respect to the final time $t_f$ as $t_f\to\infty$. Even though the negative exponential decay may only be effective for $t_f>\kappa^{-1}\sim p^8h_{\min}^{-4}$, it is nonetheless a manifestation of the concept of hypocoercivity, i.e., convergence to equilibrium despite the degenerate nature of the diffusion in Kolmogorov's equation. We shall refer to the preservation of the hypocoercivity structure by a numerical method as \emph{numerical hypocoercivity}.


\subsection{Error estimate}

We split the error as before,
\[
\rho:=u-\hat{\pi} u,\qquad \vartheta:=\hat{\pi} u -U,
\]
so that $e=\rho+\vartheta$.
Employing \eqref{eq:FEM} with $V=\vartheta$ and using \eqref{eq:weak_hypo}, respectively, gives the error equation
\begin{equation}\label{eq:error_eq}
	\begin{aligned}
		\frac{1}{2}\frac{\ud }{\ud t}\norm{\vartheta}{A}^2+a_h(\vartheta,\vartheta)
		=&\ 	\Ainner{(\hat{\pi} u)_t}{\vartheta}+a_h(\hat{\pi} u,\vartheta) - \Ainner{f}{\vartheta}-(f, \tau (V_t+ xV_y)),\\
		=&\ 	-\Ainner{\rho_t}{\vartheta}-a_h(\rho,\vartheta)=-a_h(\rho,\vartheta),
	\end{aligned}
\end{equation}
since $\Ainner{\rho_t}{\vartheta}=0$, from the definition of the $A$-orthogonal projection. Similarly, use the relation $\Ainner{\rho}{\vartheta_t}=0$, we have, respectively,
\[
	\begin{aligned}
a_h(\rho,\vartheta)
&=\  (\rho_x,\vartheta_x)-(x\rho,\vartheta_y)+(xn_2\rho,\vartheta)_{\partial_+\Omega} +(\rho_t + x\rho_y ,\tau (\vartheta_t + x\vartheta_y))\\
&
+\sum_{T\in\mathcal{T}}\Big(  -(\rho_{xx},\tau (\vartheta_t+ x\vartheta_y))_T -(\nabla \rho_{xx} , A\nabla \vartheta)_T+( \nabla (x\rho_y), A\nabla \vartheta)_T \Big)\\
=&\  (\rho_x,\vartheta_x)-(\rho, x\vartheta_y + \vartheta_t)+(xn_2\rho,\vartheta)_{\partial_+\Omega} +(\rho_t + x\rho_y ,\tau (\vartheta_t + x\vartheta_y)) \\&
+\sum_{T\in\mathcal{T}}\Big(  -(\rho_{xx},\tau (\vartheta_t+ x\vartheta_y))_T -(\nabla \rho_{xx} , A\nabla \vartheta)_T \\
&+( \nabla (x\rho_y), A\nabla \vartheta)_T
- (\nabla \rho, A \nabla \vartheta_t)_T \Big)\\
=&\  (\rho_x,\vartheta_x)-( \rho, x\vartheta_y +\vartheta_t)+(xn_2\rho,\vartheta)_{\partial_+\Omega} +(\rho_t + x\rho_y ,\tau (\vartheta_t + x\vartheta_y))  \\
&+\sum_{T\in\mathcal{T}} \Big( -(\rho_{xx},\tau (\vartheta_t+ x\vartheta_y))_T
+ (\nabla \rho_x , A\nabla \vartheta_x)_T-(n_1\nabla \rho_x  ,A\nabla \vartheta)_{\partial T} \\
&+( \nabla (x\rho_y), A\nabla \vartheta)_T
+ (\nabla \rho, A \nabla (x \vartheta_y))_T
- (\nabla \rho, A \nabla (\vartheta_t + x \vartheta_y))_T
\Big).
\end{aligned}
\]
For the eighth and ninth terms on the right-hand side of the above identity, we  have
\[
\begin{aligned}
&( \nabla (x\rho_y), A\nabla \vartheta)_T
+ (\nabla \rho, A \nabla (x \vartheta_y))_T
\\
=&( \rho_y (\nabla x), A\nabla \vartheta)_T   + ( \nabla \rho_y, A\nabla \vartheta x)_T
+ (\nabla \rho, A \nabla  \vartheta_y x)_T + (\nabla \rho, A (\nabla x) \vartheta_y))_T
\\
=&\  (\alpha \rho_y, \vartheta_x)_T+ 2(\beta \rho_y, \vartheta_y)_T  + ( \alpha \rho_x, \vartheta_y)_T + (xn_2 \nabla \rho, A \nabla  \vartheta )_{\partial T}.
\end{aligned}
\]
Using the above relation, we have
\[
	\begin{aligned}
&a_h(\rho,\vartheta)
=\  (\rho_x,\vartheta_x)-( \rho, x\vartheta_y +\vartheta_t)+(xn_2\rho,\vartheta)_{\partial_+\Omega} +(\rho_t + x\rho_y ,\tau (\vartheta_t + x\vartheta_y))  \\
&+\sum_{T\in\mathcal{T}} \Big( -(\rho_{xx},\tau (\vartheta_t+ x\vartheta_y))_T
+ (\nabla \rho_x , A\nabla \vartheta_x)_T
- (\nabla \rho, A \nabla (\vartheta_t + x \vartheta_y))_T \\
&+ (\alpha \rho_y, \vartheta_x)_T+ 2(\beta \rho_y, \vartheta_y)_T  + ( \alpha \rho_x, \vartheta_y)_T
 + (xn_2 \nabla \rho, A \nabla  \vartheta )_{\partial T}
-(n_1\nabla \rho_x  ,A\nabla \vartheta)_{\partial T}
\Big).
\end{aligned}
\]
%
For the last two terms on the right-hand side of the above identity, we  have
\[
\begin{aligned}
& (xn_2 \nabla \rho, A \nabla  \vartheta )_{\partial T} -(n_1\nabla \rho_x  ,A\nabla \vartheta)_{\partial T}
= (xn_2 \rho_{x} , \alpha \vartheta_x+\beta\vartheta_y)_{\partial T}\\
&
+ (xn_2 \rho_{y} , \beta \vartheta_x+\gamma\vartheta_y)_{\partial T} -(n_1 \rho_{xx} , \alpha \vartheta_x+\beta\vartheta_y)_{\partial T}
-(n_1 \rho_{xy} , \beta \vartheta_x+\gamma\vartheta_y)_{\partial T} \\
\le&\
C_{inv} p h_T^{-\frac{1}{2}}  \big( \norm{\alpha  x n_2 \rho_{x} }{ \partial T}
+\norm{\beta xn_2 \rho_{y} }{ \partial T}
+\norm{\alpha  n_1 \rho_{xx} }{ \partial T}
 +\norm{\beta n_1 \rho_{xy} }{ \partial T}
 \big)\norm{\vartheta_x}{T}\\
&
+C_{inv} p h_T^{-\frac{1}{2}}   \big( \norm{\beta(\gamma \delta)^{-\frac{1}{2}}  xn_2 \rho_{x} }{ \partial T}
+\norm{\sqrt{\gamma/\delta} x n_2 \rho_{y} }{ \partial T}
+ \norm{\beta(\gamma \delta)^{\frac{1}{2}}  n_1 \rho_{xx} }{ \partial T}\\
&+\norm{\sqrt{\gamma/\delta} n_1 \rho_{xy} }{ \partial T}  \big)\norm{\sqrt{\gamma \delta}\vartheta_y}{T}
\end{aligned}
\]
using the trace-inverse estimate in the last step. In addition, using the inverse inequality, definition of $\tau$ and $A$ in Lemma \ref{lemma:hypocoer}, we have

\[
\begin{aligned}
- (\nabla \rho, A \nabla (\vartheta_t + x \vartheta_y))_T
\le&
\norm{A}{L_\infty(T)}  \norm{ \nabla \rho}{T}   C_{INV}p^{2}h_T^{-1}  \norm{ (\vartheta_t+x\vartheta_y)}{T} \\
\le& \norm{( C_{INV}p^{2}h^{-1}  \norm{A}{L_\infty(T)} \tau^{-\frac{1}{2}}) \nabla \rho}{T} \norm{\sqrt{\tau} (\vartheta_t+x\vartheta_y)}{T} \\
\le&\ \norm{ C_{INV}(3C_{\delta}^T)^{-1} \nabla \rho}{T} \norm{\sqrt{\tau} (\vartheta_t+x\vartheta_y)}{T},
\end{aligned}
\]
where we have used the bound $\norm{A}{L_\infty(T)} \leq (6\delta)^{-1}$, with $\delta$ defined Lemma \ref{lemma:hypocoer}.
Performing estimation in standard fashion, and using the last bound, we have
\[
\begin{aligned}
	&a_h(\rho,\vartheta)
	\le \  \norm{\rho_x}{}\norm{\vartheta_x}{}+\norm{\tau^{-\frac{1}{2}} \rho}{}\norm{\sqrt{\tau} (x\vartheta_y + \vartheta_t)}{}
+\norm{\sqrt{xn_2}\rho}{\partial_+\Omega}\norm{\sqrt{xn_2}\vartheta}{\partial_+\Omega}\\
	&+\norm{\sqrt{\tau}(\rho_t+x\rho_y)}{}\norm{\sqrt{\tau} (\vartheta_t+x\vartheta_y)}{} + \big( \sum_{T\in\mathcal{T}}  \norm{\sqrt{\tau} \rho_{xx}}{T}^2\big)^{\frac{1}{2}} \norm{\sqrt{\tau} (\vartheta_t+x\vartheta_y)}{}
\\
	&
+\norm{\sqrt{ A}\nabla_\mathcal{T}^{} \vartheta_x}{}\norm{\sqrt{A}\nabla_\mathcal{T}^{} \rho_x }{}
+ \big( \sum_{T\in\mathcal{T}}
   \norm{ C_{INV}(3C_{\delta}^T)^{-1} \nabla \rho}{T}^2\big)^{\frac{1}{2}} \norm{\sqrt{\tau} (\vartheta_t+x\vartheta_y)}{}
\\
&
+\norm{\alpha\rho_y}{}\norm{\vartheta_x}{}
+2\norm{ \beta(\gamma\delta)^{-\frac{1}{2}}\rho_y}{} \norm{\sqrt{\gamma\delta} \vartheta_y}{} +\norm{ \alpha (\gamma\delta)^{-\frac{1}{2}}\rho_x}{} \norm{\sqrt{\gamma\delta} \vartheta_y}{}
\\
	&+
\Big(\sum_{T\in\mathcal{T}}   C_{inv}^2p^2h_T^{-1}  \big( \norm{\alpha  x n_2 \rho_{x} }{ \partial T}^2
+\norm{\beta xn_2 \rho_{y} }{ \partial T}^2  +\norm{\alpha  n_1 \rho_{xx} }{ \partial T}^2
\\
&+\norm{\beta n_1 \rho_{xy} }{ \partial T}^2 \big)\Big)^{\frac{1}{2}}
\norm{ \vartheta_x}{}
 +\Big(\sum_{T\in\mathcal{T}}    C_{inv}^2p^2h_T^{-1} \big( \norm{\beta(\gamma \delta)^{-\frac{1}{2}}  xn_2 \rho_{x} }{ \partial T}^2
  \\
&\hspace{0cm}
+\norm{\sqrt{\gamma/\delta} x n_2 \rho_{y} }{ \partial T}^2
 + \norm{\beta(\gamma \delta)^{-\frac{1}{2}}  n_1 \rho_{xx} }{ \partial T}^2
+\norm{\sqrt{\gamma/\delta} n_1 \rho_{xy} }{ \partial T}^2 \big) \Big)^{\frac{1}{2}}\norm{\sqrt{\gamma\delta}\vartheta_y}{}.
\end{aligned}
\]
Rearranging terms, we have
\[
\begin{aligned}
		&a_h(\rho,\vartheta)
	\le   \norm{\sqrt{xn_2}\rho}{\partial_+\Omega}\norm{\sqrt{xn_2}\vartheta}{\partial_+\Omega}+ \norm{\sqrt{A}\nabla_\mathcal{T}^{} \rho_x }{}\norm{\sqrt{ A}\nabla_\mathcal{T}^{} \vartheta_x}{}
+\Big( \sum_{T\in\mathcal{T}}
\big(\norm{\tau^{-\frac{1}{2}} \rho}{T}^2 \\
	&
+ \norm{\sqrt{\tau}(\rho_t+x\rho_y)}{T}^2
+ \norm{\sqrt{\tau} \rho_{xx}}{T}^2
+ \norm{ C_{INV}(3C_{\delta}^T)^{-1} \nabla \rho}{T}^2 \big) \Big)^{\frac{1}{2}} \norm{\sqrt{\tau} (\vartheta_t+x\vartheta_y)}{}\\
	&+
\Big(\sum_{T\in\mathcal{T}}   \big( \norm{\rho_{x}}{T}^2 +\norm{\alpha \rho_{y}}{T}^2 +C_{inv}^2p^2h_T^{-1}  ( \norm{\alpha  x n_2 \rho_{x} }{ \partial T}^2
+\norm{\beta xn_2 \rho_{y} }{ \partial T}^2   \\
&+ \norm{\alpha  n_1 \rho_{xx} }{ \partial T}^2
+\norm{\beta n_1 \rho_{xy} }{ \partial T}^2)\big)\Big)^{\frac{1}{2}}
\norm{ \vartheta_x}{}\\
& +  \Big(\sum_{T\in\mathcal{T}}  \big(
(\gamma\delta)^{-1} (\norm{ \alpha \rho_{x}}{T}^2
+4\norm{ \beta  \rho_{y}}{T}^2)
+ C_{inv}^2p^2h_T^{-1} ( \norm{\beta(\gamma \delta)^{-\frac{1}{2}}  xn_2 \rho_{x} }{ \partial T}^2\\
&\hspace{0cm}
+\norm{\sqrt{\gamma/\delta} x n_2 \rho_{y} }{ \partial T}^2
+\norm{\beta(\gamma \delta)^{-\frac{1}{2}}  n_1 \rho_{xx} }{ \partial T}^2
+\norm{\sqrt{\gamma/\delta} n_1 \rho_{xy} }{ \partial T}^2) \big) \Big)^{\frac{1}{2}}\norm{\sqrt{\gamma\delta}\vartheta_y}{}.
\end{aligned}
\]

Then, using the definition of $A$ and $\tau$, and the approximation \eqref{best_approx}, we deduce
\begin{equation}\label{ahstab}
\begin{aligned}
& \frac{	|a_h(\rho,\vartheta)|}{\ndg{\vartheta}}
	\le \   C_{app}\Big( \norm{\sqrt{hxn_2}}{L_\infty(\partial_+\Omega)}+\norm{\sqrt{A}h^{-1}}{L_\infty(\Omega)}+ C_{INV}^{}p^{2}
 \nonumber\\
	&\hspace{0cm}
+ C_{INV}^{-1}p^{-2}(1+\norm{xh}{L_\infty(\Omega)})
+ C_{INV}^{} (\max_{T\in\mathcal{T}} (C_\delta^T))
+ 1
	+\norm{\alpha }{L_\infty(\Omega)}
\\
	&
+C_{inv}p
\big(\norm{\alpha xn_2  h^{-1}  }{ L_\infty(\Gamma)}
+\norm{\beta xn_2 h^{-1}}{L_\infty(\Gamma)}
+\norm{\alpha h^{-2} n_1  }{ L_\infty(\Gamma)}
\\
	& +\norm{\beta h^{-2} n_1}{L_\infty(\Gamma)}\big)
+ \norm{\alpha (\gamma \delta)^{-\frac{1}{2}} }{L_\infty(\Omega)}+
\norm{\beta(\gamma\delta)^{-\frac{1}{2}}}{L_\infty(\Omega)}
\\
& C_{inv}p \big(
\norm{\beta(\gamma\delta )^{-\frac{1}{2}} h^{-1} xn_2  }{L_\infty(\Gamma)}
+\norm{\sqrt{\gamma/\delta} h^{-1} xn_2 }{L_\infty(\Gamma)}
+\norm{\beta(\gamma\delta )^{-\frac{1}{2}} h^{-2}n_1  }{L_\infty(\Gamma)}
\\
&
+\norm{\sqrt{\gamma/\delta} h^{-2}n_1 }{L_\infty(\Gamma)} \big) \Big)
\Big(\sum_{T\in\mathcal{T}} |h^{s}u|^2_{H^{s+1}(T)}\Big)^{\frac{1}{2}}
+  C_{INV}^{-1}p^{-2} \Big(\sum_{T\in\mathcal{T}} |h^{s+1}u_t|^2_{H^{s+1}(T)}\Big)^{\frac{1}{2}}\\
		\le& \   C\big(1+ \norm{xn_2}{L_\infty(\Gamma)}+\norm{xh}{L_\infty(\Omega)}\big) \Big(\sum_{T\in\mathcal{T}} |h^{s}u|^2_{H^{s+1}(T)}+ |h^{s+1}u_t|^2_{H^{s+1}(T)} \Big)^{\frac{1}{2}} ,
\end{aligned}
\end{equation}
for some $C>0$, independent of $h$, of $t$, of $\norm{x}{L_\infty(\Omega)}$, and of $u$, but dependent, in principle from $p$ and the shape-regularity of the mesh. We note that the lack of dependence on $\norm{x}{L_\infty(\Omega)}$ is due to the presence of $\norm{xn_2}{L_\infty(\partial_-T)}$ in $\delta_T$, in that the ratio
$
\norm{x}{L_\infty(T)}/\norm{xn_2}{L_\infty(\partial_-T)},
$
is uniformly bounded by a constant dependent only on the shape regularity of the mesh (perhaps times $h$).

Therefore, applying \eqref{lemma:hypocoer} and \eqref{ahstab} into \eqref{eq:error_eq}, gives
\begin{equation}\label{vartheta_stab}
	\frac{1}{2}\frac{\ud }{\ud t}\norm{\vartheta}{A}^2+\frac{1}{8}\ndg{\vartheta}^2
	\le C \sum_{T\in\mathcal{T}} \big(|h^{s}u|^2_{H^{s+1}(T)}+ |h^{s+1}u_t|^2_{H^{s+1}(T)}\big),
\end{equation}
upon absorbing the dependence on $hx$ into the generic constant $C$.

Using the calculations above, we are now ready to state an \emph{a priori} error bound.

\begin{theorem}
	Assuming that the exact solution $u\in H^{1}(I;H^{2.5+\epsilon}(\Omega))\cap \\H^{1}(I; H^{s+1}(\Omega,\mathcal{T}))$, for $1.5<s\leq p$, $p
\geq 2$, then the error of \eqref{eq:FEM} satisfies
	\begin{equation}\label{result}
		\begin{aligned}
			\norm{e(t_f)}{A}
			\le  &\	 e^{-\frac{\kappa }{8}t_f} \norm{e(0)}{A} \\
&+C\bigg(\int_0^{t_f}e^{\frac{\kappa}{8}(s-t_f)}
 \sum_{T\in\mathcal{T}}\big( |h^{s} u|_{H^{s+1}(T)}^2+|h^{s+1}u_t|_{H^{s+1}(T)}^2\big)\ud s\bigg)^{\frac{1}{2}},
		\end{aligned}
	\end{equation}
for $C>0$, independent of $t_f$, of $u$ and of $U$.
\end{theorem}
\begin{proof}
From the definition of the projection $\hat{\pi}$, we have $\Ainner{\rho_t}{\vartheta}=\Ainner{\vartheta_t}{\rho}=\Ainner{\vartheta}{\rho}=0$. Using the latter, along with Lemma \ref{hypoco_stab} and \eqref{vartheta_stab}, we have, respectively,
	\begin{equation}\label{vartheta_stabtwo}
		\begin{aligned}
		&\	\frac{1}{2}\frac{\ud}{\ud t}\norm{e(t)}{A}^2+ \frac{\kappa }{8} \norm{e(t)}{A}^2
		\\
			\le  &\
		\frac{1}{2} \frac{\ud}{\ud t}\norm{\vartheta(t)}{A}^2
			+\frac{\kappa }{4}\norm{\vartheta(t)}{A}^2
+\Ainner{\rho_t}{\rho}
			+\frac{\kappa }{4} \norm{\rho(t)}{A}^2\\
			\le &\
			C \sum_{T\in\mathcal{T}}\big( |h^{s}u|_{H^{s+1}(T)}^2+|h^{s+1}u_t|_{H^{s+1}(T)}^2
+|h^{s+3} u|^2_{H^{s+1}(T)}\big),
		\end{aligned}
	\end{equation}
	upon absorbing the dependence of $x$ in the generic constant $C$.
 Gr\"onwall's Lemma already implies the result.
\end{proof}

Comparing \eqref{result} and \eqref{SUPG_reg_error}, (or even \eqref{result_old},) we  observe that the additional ``hypocoercivity-inducing'' stabilisation results into favourable dependence of the error with respect to the final time $t_f$ as $t_f\to\infty$.
\begin{remark}[$p=1$]\label{p=1_est}
 The error bound  \eqref{result} only states the result for  polynomial order $p\geq 2$. However, we point out that for $p=1$, the proposed method will converge at an optimal order $\mathcal{O}(h)$ under the  regularity assumption $u\in H^{2.5+\epsilon}(\Omega)$.
\end{remark}

\begin{remark}\label{hpversion_est}
	It is possible to prove $hp$-version \emph{a priori} error estimates upon considering $H^1$-conforming $hp$-version best approximation estimates instead of \eqref{best_approx}. However, available $hp$-version best approximation estimates in the literature for the $L_2$-projection error in the $H^2$-seminorm are significantly suboptimal with respect to the polynomial degree $p$. To avoid dwelling in such issues, we opted for not tracking the convergence rate with respect to $p$.
\end{remark}

\section{A fully discrete hypocoercivity-exploiting method}\label{sec:fullDiscrete}
We shall now modify the spatially discrete method \eqref{eq:FEM} above and further discretise the evolutionary derivative by a discontinuous Galerkin (dG) timestepping method.

We begin by subdividing $I=(0,t_f]$, into the family $\mathcal{I} := \{ I_n : n = 1,\ldots, N\}$, where $I_n := (t_{n-1}, t_n]$ for a strictly increasing sequence $0 = t_0 < t_1 < \ldots < t_N = t_f$, and let $k:I\to\mathbb{R}_+$, with $k|_{I_n} := k_n = t_n - t_{n-1}$, the timestep.
We define space-time finite element space
\begin{equation*}
	V_{h,k}\equiv	V_{h,k}^{p,q} := \{ V \in L_2(I; V_h^p) : V|_{I_n} \in \mathbb{P}_q(I_n), n=1,\dots, N\}
\end{equation*}
for $q \in\mathbb{N}$. We denote by $\tjump{V}_{n}$ the time-jump across $t_n$, viz.,
\[
	\tjump{V}_n := V(t_n^+) - V(t_n^-) := \lim_{\epsilon \to 0^+} V(t_n+\epsilon) - \lim_{\epsilon \to 0^+} V(t_n-\epsilon)
\]
for $n=1,\dots,N$.

Starting from \eqref{eq:FEM}, we consider the following dG methods in time variable, hypocoercivity-exploiting SUPG method: find $U \in V_{h,k}$ such that
\begin{equation}\label{eq:DG_FEM}
	B(U,V)
	=
	\Ainner{U(t_0^-)}{V(t_0^+)} + \int_0^{t_f} \Big( \Ainner{f}{V} + (f,\tau(V_t + xV_y)) \Big)\ud t
\end{equation}
for all $V \in V_{h,k}$, whereby
	\begin{equation}\label{space-time_method}
		\begin{aligned}
			B(U,V)
	:=&
	\sum_{n=1}^N \int_{I_n} \left( \Ainner{U_t}{V} + a_h(U,V) \right) \ud t + \sum_{n=2}^{N} \Ainner{\tjump{U}_{n-1}}{V( t_{n-1}^+) }\\
&+	\Ainner{U(t_0^+)}{V(t_0^+)} ,
		\end{aligned}
	\end{equation}
with $U(t_0^-) := \hat{\pi} u_0$, for $\hat{\pi} \colon H^1_-(\Omega) \to V_h$ A-orthogonal projection operator.

Of course, in practice \eqref{eq:DG_FEM} is typically solved on each time-step, viz.,
\begin{equation}\label{eq:DG_FEM_single_interval}
	\begin{aligned}
&	\int_{I_n} \left( \Ainner{U_t}{V} + a_h(U,V) \right) \ud t + \Ainner{U(t_{n-1}^+)}{V( t_{n-1}^+) }\\
	=&\
	\Ainner{U(t_{n-1}^-)}{V( t_{n-1}^+) }
	+
	\int_{I_n}\Big( \Ainner{f}{V} + (f,\tau (V_t+xV_y)) \Big) \ud t,
\end{aligned}
\end{equation}
for all $V \in V_{h,k}$, $n = 1,...,N$. The SUPG parameter  $\tau$ is defined as in \eqref{def: delta and tau} for all $T\in \mathcal{T}$ on $I_n$.

\subsection{Numerical hypocoercivity}
To study the numerical hypocoercivity structure of the space-time method \eqref{space-time_method}, we set $V=U$ and $f=0$  into \eqref{eq:DG_FEM_single_interval}  and, following standard arguments, (see, e.g., \cite[Chapter 12]{thomee},) we have
\[
\frac{1}{2}\norm{U(t_n^-)}{A}^2+\frac{1}{2}\norm{U(t_{n-1}^+)}{A}^2-	\Ainner{U(t_{n-1}^-)}{U( t_{n-1}^+) } +\int _{I_n}a_h(U,U)\ud t=0,
\]
or
\begin{equation}\label{dg_coer}
\frac{1}{2}\norm{U(t_n^-)}{A}^2+\frac{1}{2}\norm{\tjump{U}_{n-1}}{A}^2+\int _{I_n}a_h(U,U)\ud t=\frac{1}{2}\norm{U(t_{n-1}^-)}{A}^2.
\end{equation}
From Lemmata \ref{lem:coercivity} and \ref{hypoco_stab}, we deduce
\begin{equation}\label{dG_stab}
\norm{U(t_n^-)}{A}^2+\norm{\tjump{U}_{n-1}}{A}^2+\frac{\kappa}{2}\int _{I_n}\norm{U}{A}^2\ud t
\le \norm{U(t_{n-1}^-)}{A}^2.
\end{equation}
Next, we manipulate the third term on the left-hand side of the last inequality to ensure a non-trivial spectral gap.
Employing the $hp$-version trace inverse estimate \cite{WHinverse}, we have
\begin{equation}\label{inverse_trace_in_time}
k_n\norm{U(t_{n}^-)}{A}^2\le  (q+1)^2\int _{I_n}\norm{U}{A}^2\ud t,
\end{equation}

Using the above estimate, we deduce
\[
\begin{aligned}
\frac{\kappa	k_n}{2 (q+1)^2}\norm{U(t_{n}^-)}{A}^2
	\le &\ \frac{\kappa}{2}\int _{I_n}\norm{U}{A}^2\ud t,
\end{aligned}
\]
which upon substitution into \eqref{dG_stab} gives
\[
\Big(1+\frac{\kappa	k_n}{2(q+1)^2} \Big)\norm{U(t_n^-)}{A}^2\le \norm{U(t_{n-1}^-)}{A}^2,
\]
we deduce
\begin{equation}\label{dG_stab_two}
\norm{U(t_f)}{A}^2\le \prod_{n=1}^N	\Big(1+\frac{\kappa	k_n}{2 (q+1)^2}	\Big)^{-1}\norm{U(t_{0}^-)}{A}^2 ,
\end{equation}
which is, again, a manifestation of ``numerical'' hypocoercivity.

\begin{remark}
From \eqref{dG_stab_two}, assuming constant timestep $k_n=k$ for all $n=1,\dots,N$, we deduce
\begin{equation}\label{eq:convergenceToEquilibriumForFullyDiscrete}
\norm{U(t_f)}{A}\le 	\Big(1+\frac{\kappa	k}{2 (q+1)^2}	\Big)^{-\frac{t_f}{2k}}\norm{U(t_{0}^-)}{A}\to  e^{-\frac{\kappa}{4(q+1)^{2}}t_f} \norm{U(t_{0}^-)}{A},
\end{equation}
as $k\to 0$. Comparing with \eqref{hypoSUPG_stab}, the exponent has the same scale of the respective one in the semidiscrete case for $q=0$.
\end{remark}

\subsection{Error estimate}
We define the space-time norm
\begin{equation}
	\ndg{U}_{\rm st}^2
	:=
  \frac{1}{2}\Big( \sum_{n=1}^{N-1} \norm{\tjump{U}_{n}}{A}^2
+  \norm{U(t_N^-)}{A}^2 + \norm{U(t_0^+)}{A}^2
\Big)
+\sum_{n = 1}^n \int_{I_n}  \frac{1}{4} \ndg{U}^2 \ud t.
\end{equation}
Set $e := u-U$. The consistency of the method implies the error equation:
\begin{equation}\label{eq:ErrorEquation}
 \sum_{n = 1}^n   \int_{I_n} \left( \Ainner{e_t}{V} + a_h (e,V) \right) \ud t
    +
\sum_{n = 2}^n    \Ainner{\tjump{e}_{n-1}}{V(t_{n-1}^+)}
+
 \Ainner{e_{0}^+}{V(t_{0}^+)}
    =
    0,
\end{equation}
for all $V\in V_{h,k}$.

\begin{definition}
	We define the space-time projection $\pi_{\rm st}:H^1(I_n;H^1(\Omega)) \to V_{h,k}^{p,q}$ constructed for each $I_n$, $n=1,\dots,N$, as the tensor product  $\pi_{\rm st}: = \tilde{\pi}^t\circ \hat{\pi} = \hat{\pi} \circ \tilde{\pi}^t$, whereby
$\hat{\pi} \colon H^1(\Omega) \to V_h$ is the $A$-orthogonal spatial projection, defined by
		$
			\Ainner{\hat{\pi} v}{V} = \Ainner{v}{V}
		$
		for all $V \in V_h$, and $\tilde{\pi}^t|_{I_n} \colon H^1(I_n) \to \mathbb{P}_q(I_n)$, such that
$	(\tilde{\pi}^t v )(t_n^-)
=
v (t_n^-),$ and	\[
				\int_{I_n} {\tilde{\pi}^t v}{w} \ud t
				=
				\int_{I_n} {v}{w} \ud t,
\]
		for all $w \in \mathbb{P}_{q-1}(I_n)$.

\end{definition}
We comment on the approximation properties and stability of $\tilde{\pi}^t$, confining ourselves to $h$-version estimates; see \cite{SchoetzauSchwab00,thomee} for details. For any function $v\in H^{r}(I_n)$ with $1\leq r \leq q+1$, we have
\begin{equation}\label{eq:t projection appromixation}
 \norm{\partial_t^m (v-{\tilde{\pi}^t v})}{I_n} \leq  C_{app} \norm{k^{r-m} \partial_t^{r} v}{I_n},  \qquad  m=0,1
\end{equation}
and
\begin{equation}\label{eq:t projection appromixation on trace}
 |(v-{\tilde{\pi}^t v})(t_{n-1}^+)| \leq  C_{app} \norm{k^{r-\frac{1}{2}}  \partial_t^{r} v}{I_n},
\end{equation}
and the stability
\begin{equation}\label{eq:t projection stability}
 \norm{{\tilde{\pi}^t v}}{I_n} \leq   \norm{{v}}{I_n} + C_{app} \norm{k \partial_t v}{I_n},
\end{equation}
where  constant $C_{app}$ is independent of $k$.

Combining \eqref{eq:t projection appromixation}, \eqref{eq:t projection stability} and the approximation results \eqref{best_approx}, we can deduce the following relation:

\begin{equation}\label{eq:space_time_approximation}
	\begin{aligned}
\int_{I_n}
\norm{\nabla^\lambda (v-\pi_{\rm st}v)}{T}^2
\ud t
\le&  \tilde{C}_{app} \int_{I_n}
\big(|  k^{r}  \partial_t^r  v|_{H^{\lambda}(T)}^2
+ |h^{s+1-\lambda}v|_{H^{s+1}(T)}^2 \\
&\hspace{1cm}+|h^{s+1-\lambda}k \partial_t v|_{H^{s+1}(T)}^2
\big)\ud t.
	\end{aligned}
\end{equation}
Similarly, using the trace inequality, we deduce
\begin{equation}\label{eq:space_time_approximation trace}
	\begin{aligned}
\int_{I_n}
\norm{h^{\frac{1}{2}}\nabla^\lambda (v-\pi_{\rm st}v)}{\partial T}^2
\ud t
\le&  \tilde{C}_{app} \int_{I_n}
\big(|  k^{r}  \partial_t^r  v|_{H^{\lambda}(\partial T)}^2
+ |h^{s+1-\lambda}v|_{H^{s+1}( T)}^2 \\
&\hspace{1cm}+|h^{s+1-\lambda}k \partial_t v|_{H^{s+1}(T)}^2
\big)\ud t ,
	\end{aligned}
\end{equation}
for any functions $v$ satisfying $v\in H^1(I_n;H^{s+1}(\Omega,\mathcal{T})) \cap H^{r}(I_n;H^{2.5+\epsilon}(\Omega))$, with arbitrary small $\epsilon>0$, $1\le r\le q+1$, $1.5< s\le p$, $p\geq2$ and $\lambda\leq 2$.

We now split the error as $e = \eta + \xi$, where $\eta = u-\pi_{\rm st} u$ and $\xi = \pi_{\rm st} u - U$. From \eqref{eq:ErrorEquation}, along with completely analogous elementary manipulations to the derivation of \eqref{dg_coer}, we arrive at
\begin{equation}\label{eq:errorEquationSplit}
  \begin{aligned}
  &\hspace{0cm} \frac{1}{2} \Big( \norm{\xi(t_N^-)}{A}^2 +  \norm{\xi(t_0^+)}{A}^2 +   \sum_{n = 2}^N   \norm{\tjump{\xi}_{n-1}}{A}^2 \Big)
  +
\sum_{n = 1}^N   \int_{I_n} a_h (\xi,\xi) \ud t\\
    =&\
    -\sum_{n = 1}^N \int_{I_n} \left( \Ainner{\eta_t}{\xi} + a_h(\eta,\xi) \right) \ud t
    -
    \sum_{n = 2}^N\Ainner{\tjump{\eta}_{n-1}}{\xi(t_{n-1}^+)}
    - \Ainner{\eta (t_{0}^+)}{\xi(t_{0}^+)},
\end{aligned}
\end{equation}
upon setting $V=\xi$.

A key motivation for using the projection $\pi_{\rm st}$ is the following orthogonality result.
\begin{lemma}
    For all $V \in V_h$ and $n=1,\dots,N$, it holds
    \begin{equation}\label{eq:timeDerivativeCalc1}
        -\int_{I_n} \Ainner{\eta_t}{V} \ud t
        =
        \Ainner{\eta(t_{n-1}^+)}{V(t_{n-1}^+)}.
    \end{equation}
\end{lemma}
\begin{proof}
By construction of $\pi_{\rm st}$, we have
\[
 \Ainner{\eta(t_n^-)}{V(t_n^-)}= \Ainner{u(t_n^-) - \hat{\pi} u(t_n^-)}{V(t_n^-)}=0
 \qquad\text{and}\qquad
   \int_{I_n} \Ainner {\eta}{W} \ud t=0,
 \]
whenever $W\in \mathbb{P}_{q-1}(I_n)\times V_h$. Hence, since $V_t\in  \mathbb{P}_{q-1}(I_n)\times V_h$, integration by parts on the left-hand side of \eqref{eq:timeDerivativeCalc1} already yields the result.
\end{proof}

Inserting, now, \eqref{eq:timeDerivativeCalc1} into \eqref{eq:errorEquationSplit}, we get
\[
	\begin{aligned}
		&\frac{1}{2} \Big( \norm{\xi(t_N^-)}{A}^2 +  \norm{\xi(t_0^+)}{A}^2 +   \sum_{n = 2}^N   \norm{\tjump{\xi}_{n-1}}{A}^2 \Big)\!
  +
\sum_{n = 1}^N   \int_{I_n} a_h (\xi,\xi) \ud t
= 	\! \sum_{n = 1}^N 	\int_{I_n} \!  a_h(\eta,\xi)  \ud t,
	\end{aligned}
\]
or, using \eqref{lemma:hypocoer} and standard manipulations
\begin{equation}
	\begin{aligned}
	\ndg{\xi}_{\rm st}^2
		\le \! \sum_{n = 1}^N\int_{I_n}  a_h(\eta,\xi)\ud t.
	\end{aligned}
\end{equation}
Working as in the proof of \eqref{ahstab}, we have
\[
\begin{aligned}
\frac{	|a_h(\eta,\xi)|}{\ndg{\xi}}
	\le &\  \bigg(\norm{\sqrt{xn_2}\eta}{\partial_+\Omega}^2+ \norm{\sqrt{A}\nabla_\mathcal{T}^{} \eta_x }{}^2+
\norm{\tau^{-\frac{1}{2}} \eta }{}^2
+\norm{\sqrt{\tau}(\eta_t+x\eta_y)}{}^2
\\
	&+
\sum_{T\in\mathcal{T}} \Big(
\norm{\sqrt{\tau}\eta_{xx}}{T}^2
+ \norm{ C_{INV}(3C_{\delta}^T)^{-1} \nabla \eta}{T}^2
+ \norm{\eta_x}{T}^2+\norm{\alpha\eta_y}{T}^2
 \\
	& +C_{inv}^2p^2h_T^{-1}
\big( \norm{\alpha  x n_2 \eta_{x} }{ \partial T}^2
+\norm{\beta xn_2 \eta_{y} }{ \partial T}^2
+\norm{\alpha  n_1 \eta_{xx} }{ \partial T}^2
+\norm{\beta  n_1 \eta_{xy}}{ \partial T}^2\big)
	\\
	&+\norm{\alpha(\gamma \delta)^{-\frac{1}{2}} \eta_x}{T}^2
  + 4\norm{ \beta(\gamma\delta)^{-\frac{1}{2}}\eta_y}{T}^2\\
& + C_{inv}^2p^2h^{-1} \big( \norm{\beta(\gamma\delta )^{-\frac{1}{2}} xn_2 \eta_{x} }{\partial T}^2
+\norm{\sqrt{\gamma/\delta } xn_2 \eta_{y}}{\partial T}^2
\\
&+\norm{\beta(\gamma\delta )^{-\frac{1}{2}} n_1 \eta_{xx} }{\partial T}^2
+\norm{\sqrt{\gamma/\delta } n_1 \eta_{xy}}{\partial T}^2 \big)\Big)\bigg)^\frac{1}{2}.
\end{aligned}
\]

Then, using the definition of A, and the approximation \eqref{eq:space_time_approximation}, we deduce
\begin{equation}\label{eq:continuity_a_FD}
	\begin{aligned}
	\int_{I_n} |a_h(\eta,\xi)|  \ud t \leq   \tilde{C}_{app} \Big(	\int_{I_n}  \sum_{T\in\mathcal{T}} \big(\mathcal{T}_T + \mathcal{S}_T\big) \ud t \Big)^{\frac{1}{2}} \Big( \int_{I_n} \frac{1}{4} \ndg{\xi}^2 \ud t \Big)^{\frac{1}{2}},
	\end{aligned}
\end{equation}
where
\begin{equation*}\label{def: time error bound}
	\begin{aligned}
& \mathcal{T}_T: =
\norm{\sqrt{xn_2}}{L_{\infty}(\partial_T\cap \partial_+\Omega)}^2\norm{k^r \partial^r_t u}{\partial T\cap \partial_+\Omega }^2
+\norm{\sqrt{A}}{L_{\infty}(T)}^2\norm{k^r \partial^r_t (\nabla u_x)}{T}^2
\\
& +\tau_T^{-1}\norm{k^r \partial^r_t  u}{T}^2
 +\tau_T \norm{k^{r-1} \partial^r  u}{T}^2
+\tau_T  \norm{x}{L_{\infty}(T)}^2 \norm{k^{r} \partial^r_t  u_y}{T}^2
+\tau_T \norm{k^{r} \partial^r_t  u_{xx}}{T}^2 \\
& +  \norm{C_{INV}(C_{\delta}^T)^{-1}}{L_{\infty}(T)}^2  \norm{k^{r} \partial^r_t  (\nabla u)}{T}^2
+   \norm{k^{r} \partial^r_t  u_x}{T}^2
+  \alpha^2  \norm{k^{r} \partial^r_t  u_y}{T}^2 \\
& +C_{inv}^2p^2h_T^{-1}
\big( \norm{\alpha  x n_2}{ L_{\infty}(\partial T)}^2  \norm{k^{r} \partial^r_t  u_x}{\partial T}^2
+\norm{\beta xn_2 }{L_{\infty}(\partial T)}^2  \norm{k^{r} \partial^r_t  u_y}{\partial T}^2 \\
& \hspace{1cm} +\norm{\alpha  n_1}{L_{\infty}(\partial T)}^2   \norm{k^{r} \partial^r_t  u_{xx}}{\partial T}^2
+\norm{\beta n_1 }{L_{\infty}(\partial T)}^2  \norm{k^{r} \partial^r_t  u_{xy}}{\partial T}^2
\big)
\\
&+ \norm{\alpha {\gamma \delta}^{-\frac{1}{2}}}{L_{\infty}(T)}^2   \norm{k^{r} \partial^r_t  u_{x}}{T}^2
+ \norm{\beta {\gamma \delta}^{-\frac{1}{2}}}{L_{\infty}(T)}^2   \norm{k^{r} \partial^r_t  u_{y}}{T}^2 \\
& +C_{inv}^2p^2h_T^{-1}  \big(
\norm{\beta(\gamma\delta )^{-\frac{1}{2}} xn_2 }{L_{\infty}(\partial T)}^2 \norm{k^{r} \partial^r_t  u_x}{\partial T}^2
+\norm{\sqrt{\gamma/\delta } xn_2 }{L_{\infty}(\partial T)}^2  \norm{k^{r} \partial^r_t  u_y}{\partial T}^2
\\
	&\hspace{1cm}
+\norm{\beta(\gamma\delta )^{-\frac{1}{2}} n_1  }{L_{\infty}(\partial T)}^2 \norm{k^{r} \partial^r_t  u_{xx}}{\partial T}^2
 +\norm{\sqrt{\gamma/\delta } n_1 }{L_{\infty}(\partial T)}^2  \norm{k^{r} \partial^r_t  u_{xy}}{\partial T}^2
\big),
	\end{aligned}
\end{equation*}
and
\begin{equation*}\label{def: space error bound}
	\begin{aligned}
& \mathcal{S}_T: = \big(
\norm{\sqrt{xn_2h}}{L_{\infty}(\partial_T\cap \partial_+\Omega)}^2
+\norm{\sqrt{A}h^{-1}}{L_{\infty}(T)}^2
 +\tau_T^{-1} h^2_T
+\tau_T  \norm{x}{L_{\infty}(T)}^2
\\
&
+\tau_T h^{-2}_T
+  \norm{C_{INV}(C_{\delta}^T)^{-1}}{L_{\infty}(T)}^2
+  \!1
+  \alpha^2
+C_{inv}^2p^2
\big( \norm{\alpha  x n_2 h^{-1} }{ L_{\infty}(\partial T)}^2 \\
	&
+\norm{\beta xn_2 h^{-1} }{L_{\infty}(\partial T)}^2
 +\norm{\alpha  n_1 h^{-2} }{L_{\infty}(\partial T)}^2
 +\norm{\beta n_1 h^{-2} }{L_{\infty}(\partial T)}^2
\big)\\
&\hspace{0cm}
+ \norm{\alpha {\gamma \delta}^{-\frac{1}{2}}}{L_{\infty}(T)}^2
+ \norm{\beta {\gamma \delta}^{-\frac{1}{2}}}{L_{\infty}(T)}^2
 +C_{inv}^2p^2 \big(
\norm{\beta(\gamma\delta )^{-\frac{1}{2}} xn_2 h_T^{-1} }{L_{\infty}(\partial T)}^2 \\
&
+\norm{\sqrt{\gamma/\delta } xn_2 h_T^{-1} }{L_{\infty}(\partial T)}^2
+\norm{\beta(\gamma\delta )^{-\frac{1}{2}} n_1  h_T^{-2} }{L_{\infty}(\partial T)}^2
\\
	&\hspace{0cm}
 +\norm{\sqrt{\gamma/\delta } n_1 h_T^{-2} }{L_{\infty}(\partial T)}^2
\big)\big)
\big(
|h^s u|^2_{H^{s+1}(T)} + |h^{s}k  \partial_t u|^2_{H^{s+1}(T)}
\big)
 +  \tau_T  |h^{s+1}  \partial_t u|^2_{H^{s+1}(T)}
. 	\end{aligned}
\end{equation*}
Then, using the definition of $A$, $\tau$, \eqref{eq:continuity_a_FD} and $\sum_{n=1}^N\int_{I_n} \frac{1}{4} \ndg{\xi}^2 \ud t \leq 	\ndg{\xi}_{\rm st}^2$,  we deduce
\begin{equation}\label{eq: fully error bound}
	\ndg{\xi}_{\rm st}^2
\leq C \sum_{n = 1}^N \mathcal{E}_n,
\end{equation}
with $\mathcal{E}_n$ given by
\begin{equation}\label{eq: fully error bound Error time step}
	\begin{aligned}
		\mathcal{E}_n := & \!   \int_{I_n}  \sum_{T\in\mathcal{T}}
\Big(
k_T^{2r}
\big(\norm{ \partial^r_t u}{\partial T\cap \partial_+\Omega }^2
+(1+ h_T^{-2} + (h_T/k_n)^2)\norm{\partial^r_t u}{T}^2
\\
&\hspace{0cm}
+ \norm{ \partial^r_t \nabla u}{T}^2
+ \norm{h^{\frac{1}{2}}\partial^r_t \nabla u}{\partial T}^2
+\norm{h \partial^r_t \nabla^2 u}{T}^2
+ \norm{h^{\frac{3}{2}} \partial^r_t \nabla^2 u}{\partial T}^2
\big)\\
&+ h_T^{2s} \big(
| u|^2_{H^{s+1}(T)} + |k  \partial_t u|^2_{H^{s+1}(T)}
\big)
\Big)
\ud t,
	\end{aligned}
\end{equation}
where $C>0$, independent of $h_T$, $k_n$, and of $u$, but dependent on $p$, $q$, the shape-regularity of the mesh, $\norm{x}{L_\infty(\Omega)}$ and $\norm{xn_2}{L_\infty(\partial_-T)}$.  

\begin{remark}\label{rem:optimalityOfhk}
The fully discrete error bound \eqref{eq: fully error bound} is optimal in terms of spatial mesh size $h$ in the space-time norm. Moreover, it is optimal in terms of time step $k$ under the condition that $h \approx \mathcal{O}(k^{\frac{1}{2}})$. For the sufficiently regular solution $u$,  the optimal error bound with quasi-uniform spatial and temporal meshes are given below with $s=p$ and $r=q+1$,
\[
	\ndg{\xi}_{\rm st}^2
\leq C \big(h^{2p} + k^{2q+1} \big).
\]
Moreover, from the triangle inequality and best approximation, we deduce the following (optimal) \emph{a priori} error estimate
\begin{equation}\label{error bound Space-Time DG norm}
	\ndg{u-U}_{\rm st}^2
\leq C \big(h^{2p} + k^{2q+1} \big).
\end{equation}
\end{remark}

We are now ready to state an \emph{a priori} error bound for the fully discrete scheme.
\begin{theorem}
	Assuming that the exact solution $u|_{I_n}\in H^1(I_n;H^{s+1}(\Omega,\mathcal{T})) \cap H^{r}(I_n;H^{2.5+\epsilon}(\Omega))$, with $1\le r\le q+1$, $1.5< s\le p$, $p\geq2$, $1\leq n\leq N$ and arbitrary small $\epsilon>0$, then the error of \eqref{eq:DG_FEM} satisfies
	\begin{equation}\label{result fully discrete}
		\begin{aligned}
			\norm{e(t_f)}{A}^2
			\le \prod_{n=1}^{N}\big( 1+\mu_n \big)^{-1} 	\norm{u_0 - \hat{\pi}u_0 }{A}^2    + C\sum_{n=1}^{N-1}  \prod_{m=n+1}^{N}\big( 1+\mu_m \big)^{-1}  \mathcal{E}_n + C\mathcal{E}_N ,
		\end{aligned}
	\end{equation}
for  $\mu_n := \kappa k_n/(4(q+1)^2)$, $\kappa\equiv\kappa(h,p) = c_{\rm hc}h_{\min}^4p^{-8}$ (as above), $\mathcal{E}_n$ defined in \eqref{eq: fully error bound Error time step} and  a constant  $C>0$, independent of $t_f$, of $u$ and of $U$.
\end{theorem}
\begin{proof}
First, using the dG-method \eqref{eq:DG_FEM_single_interval} on each time interval $I_n$, with $n=1,\dots, N$, and choosing $V= \xi$, we have the following error equation,
\begin{equation}\label{eq:ErrorEquation on each step}
  \int_{I_n} \left( \Ainner{e_t}{\xi} + a_h (e,\xi) \right) \ud t
    +
  \Ainner{\tjump{e}_{n-1}}{\xi(t_{n-1}^+)}
    =
    0.
\end{equation}
Then, using the standard manipulation and  \eqref{eq:timeDerivativeCalc1} of  projection $\pi_{st}$, we have
	\begin{equation*}
		\begin{aligned}
	&		\frac{1}{2}\norm{\xi(t_n^-)}{A}^2+\frac{1}{2}\norm{\tjump{\xi}_{n-1}}{A}^2 -\frac{1}{2}\norm{\xi(t_{n-1}^-)}{A}^2 +\int _{I_n}a_h(\xi,\xi)\ud t
&
  =  - \int_{I_n}  a_h (\eta,\xi) \ud t.
		\end{aligned}
	\end{equation*}
Next, applying Lemma \ref{lemma:hypocoer} and Lemma \ref{hypoco_stab}, we have
\[
	\frac{1}{2}\norm{\xi(t_n^-)}{A}^2 -\frac{1}{2}\norm{\xi(t_{n-1}^-)}{A}^2 +\frac{1}{4}\int _{I_n}\ndg{\xi}^2\ud t \leq
   \int_{I_n}  \frac{a_h (\eta,\xi)}{\ndg{\xi}} \ndg{\xi} \ud t.
\]
 Then, standard estimates and Lemma \ref{hypoco_stab} give
\[
\frac{1}{2}\norm{\xi(t_n^-)}{A}^2 -\frac{1}{2}\norm{\xi(t_{n-1}^-)}{A}^2 +  \frac{\kappa}{8}\int _{I_n} \norm{\xi}{A}^2\ud t \leq
 C\mathcal{E}_n.
\]
Applying the trace inverse inequality \eqref{inverse_trace_in_time} and setting $\mu_n := \kappa k_n/(4 (q+1)^2)$, we have
	\begin{equation}\label{Grownwall I_n}
		\begin{aligned}
			\big(1+\mu_n \big)\norm{\xi(t_n^-)}{A}^2  - \norm{\xi(t_{n-1}^-)}{A}^2\leq
  C\mathcal{E}_n.
		\end{aligned}
	\end{equation}
Next, using $\Ainner{\eta(t_n^-)}{\xi(t_n^-)} = \Ainner{\eta(t_{n-1}^-)}{\xi(t_{n-1}^-)} =0$ along with \eqref{Grownwall I_n}, we have
	\begin{equation*}
		\begin{aligned}
	&		\big(1+\mu_n \big)\norm{e(t_n^-)}{A}^2  - \norm{e(t_{n-1}^-)}{A}^2 \\
=&\ 	\big(1+\mu_n \big) \norm{\eta(t_n^-)}{A}^2- \norm{\eta(t_{n-1}^-)}{A}^2
+ 	\big(1+\mu_n \big)\norm{\xi(t_n^-)}{A}^2  - \norm{\xi(t_{n-1}^-)}{A}^2  \\
 \leq&\ \big(1+\mu_n \big) \norm{\eta(t_n^-)}{A}^2
+C\mathcal{E}_n.
		\end{aligned}
	\end{equation*}
By the best approximation results \eqref{eq:space_time_approximation trace}, we have that $ \norm{\eta(t_{n}^-)}{A}$ can be bounded by $\mathcal{E}_n$, thereby, concluding
 	\begin{equation*}
		\begin{aligned}
\big(1+\mu_n \big)\norm{e(t_n^-)}{A}^2
 \leq
 \norm{e(t_{n-1}^-)}{A}^2
+C \big(1+\mu_n \big) \mathcal{E}_n.
		\end{aligned}
	\end{equation*}
The last recurrence relation for $n=1,\dots,N$, implies
	\begin{equation}\label{Grownwall discrete}
		\begin{aligned}
			\norm{e(t_N^-)}{A}^2  \leq  \prod_{n=1}^{N}\big( 1+\mu_n \big)^{-1}  \norm{e(t_{0}^-)}{A}^2
+ C \sum_{n=1}^{N-1}  \prod_{m=n+1}^{N}\big( 1+\mu_m \big)^{-1}  \mathcal{E}_n + C\mathcal{E}_N .
		\end{aligned}
	\end{equation}
Finally, using the identity $e(t_{0}^-)=u_0 - \hat{\pi}u_0 $, \eqref{result fully discrete} is proven.
\end{proof}

\section{Numerical experiments}
We now present a basic numerical experiment which aims to verify the rates predicted by the theory.
To begin with, we measure semi-discrete error which was estimated in \eqref{result}.
This is followed by measuring the error for an instationary problem which was estimated in \eqref{result fully discrete}.
To complete, we demonstrate the exponential decay of the $A$-norm of a problem with no forcing which was considered in \eqref{eq:convergenceToEquilibriumForFullyDiscrete}.

For the \rev{first and second examples} of the experiments, we take $\Omega = (0,1)^2$.
This will be discretised by a sequence of successively refined uniform triangular meshes consisting of 32, 128, 512, 2048, and 8192 elements. The convergence rate of the error will be presented as a function of mesh size $h$ and of ${\rm dof}$, where dof denotes the to total number of space time basis functions.


\subsection{Semi-discrete error}
For this we do not partition the time interval $(0,t_f)$ and use a single polynomial of order $q=0$ on this.
We measure the error $\ndg{e}_{\rm st}$ defined in \eqref{error bound Space-Time DG norm}.
By the stationary properties of the problem, this should not be dissimilar to the error $\ndg{e(t_n^-)}$.
We set $t_f =0$ and choose $f$ such that $u(t,x,y) = \sin^2(\pi x)  \sin(\pi y)$.
This results in the graph in Figure \ref{fig:stationaryExperiment}. The optimal convergence rate  $\mathcal{O}(h^p)$ for $p=1,2,3,4$ is observed.
\begin{figure}\centering
\includegraphics[width = .48\linewidth]{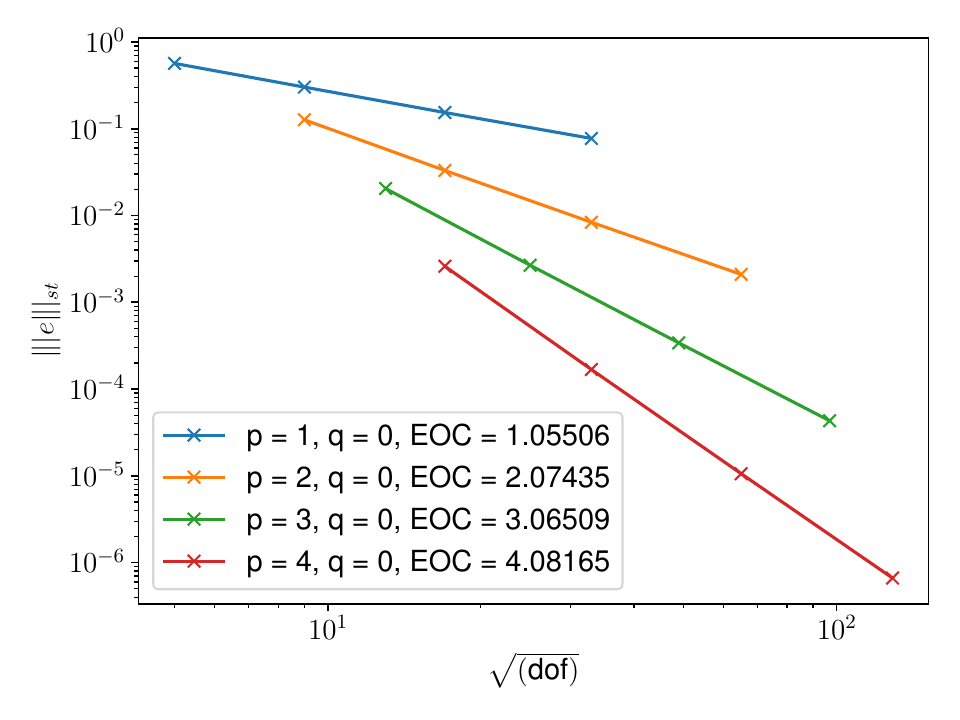}
\includegraphics[width = .48\linewidth]{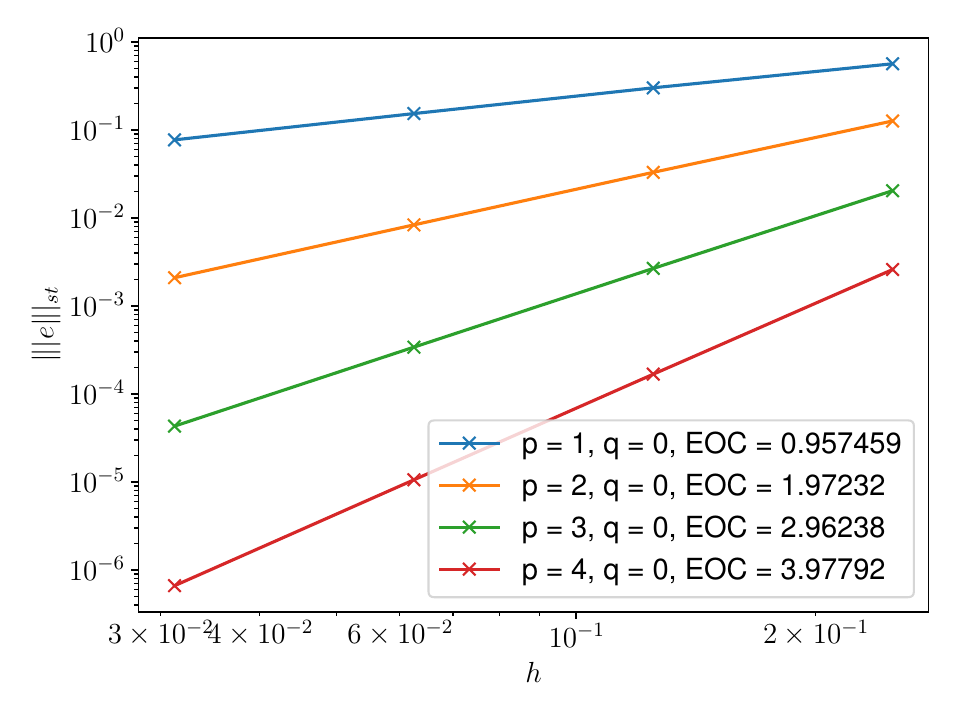}
\caption{Experiment for the semi-discrete error.}\label{fig:stationaryExperiment}
\end{figure}

\subsection{Fully-discrete error}
For this experiment, we consider $k = h^2$ as mentioned in Remark \ref{rem:optimalityOfhk}.
The error is again measured by $\ndg{e}_{\rm st}$.
Again we set $t_f =1$.
Here, we choose $f$ such that
\[ u(t,x,y) = \exp\left(-(x-0.5)^2 - (y-0.5)^2 \right) \sin^2(\pi x) \sin(\pi (y-t-0.5)) / (2-t) + 1. \]
This satisfies a relatively inhomogeneous boundary condition on $y= 0$.
The results of the experiment are provided in Figure \ref{fig:instationaryExperiment}. The optimal convergence rate $\mathcal{O}(h^p)$ for $p=1,2,3,4$ is also observed  upon  selecting polynomial order $q = 0,1,2,2$.
\begin{figure}\centering
\includegraphics[width = .48\linewidth]{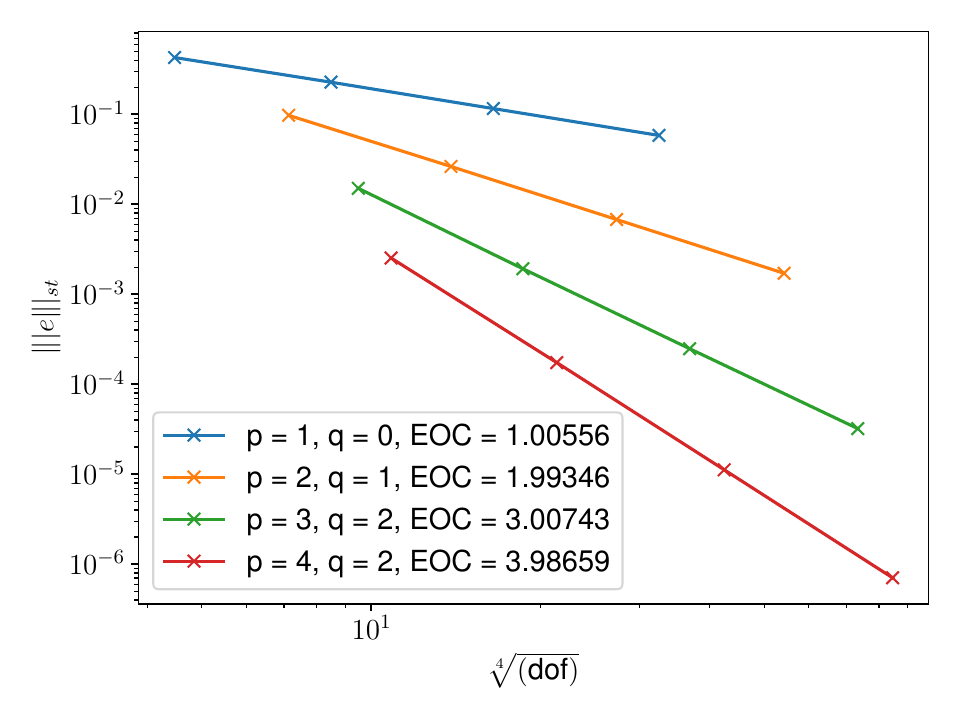}
\includegraphics[width = .48\linewidth]{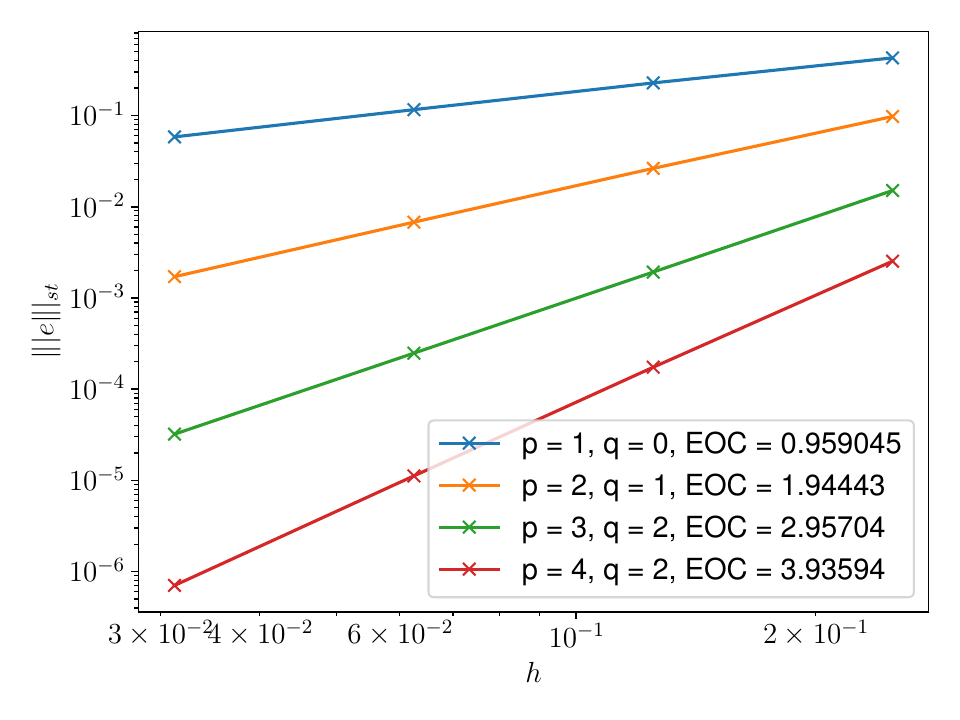}
\caption{Experiment for the fully-discrete error.}\label{fig:instationaryExperiment}
\end{figure}

\subsection{Exponential convergence to equilibrium}
%

For this experiment, \rev{we take $\Omega = (-\frac{1}{2},\frac{1}{2})^2$, $f \equiv 0$ and we set $t_f = 100$.
We consider $k = h$.}
Rather than measuring error, we plot the value of $\norm{U(t^-)}{A}$ for the points $t$ which make up the right most end points of each interval, as well as the initial condition.
We start with the initial condition \rev{$u_0(x,y) = \max(0, \frac{1}{4}- \max (|x|, |y|))$} which corresponds to a hat-type function in the middle of the domain.
This results in the graphs in Figure \ref{fig:convergenceToEquilibrium}.
\begin{figure}\centering
\includegraphics[width = .48\linewidth]{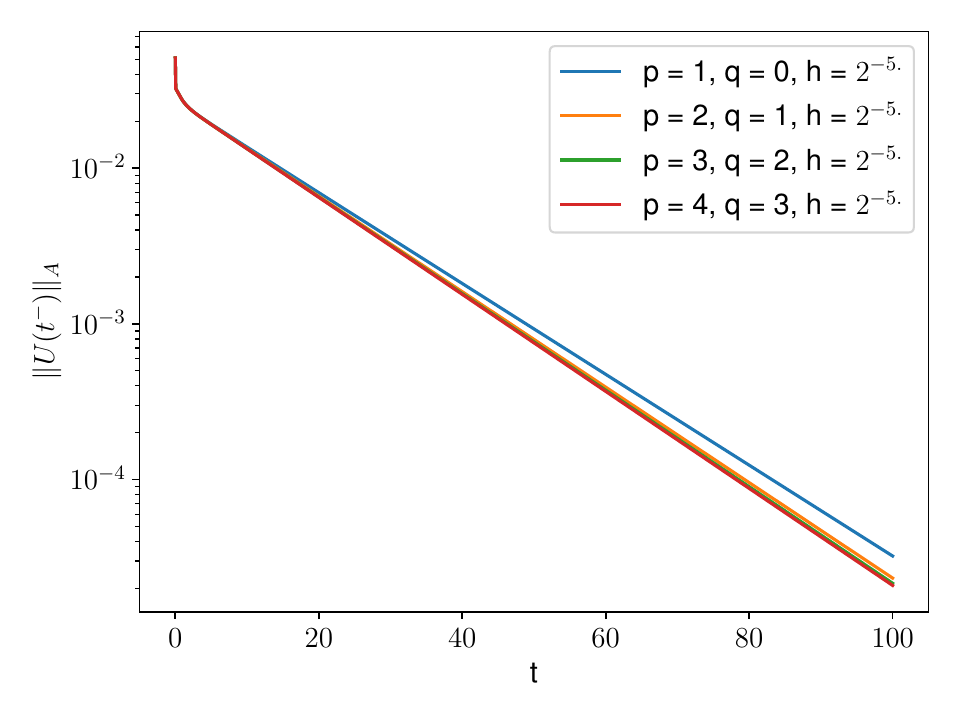}
\includegraphics[width = .48\linewidth]{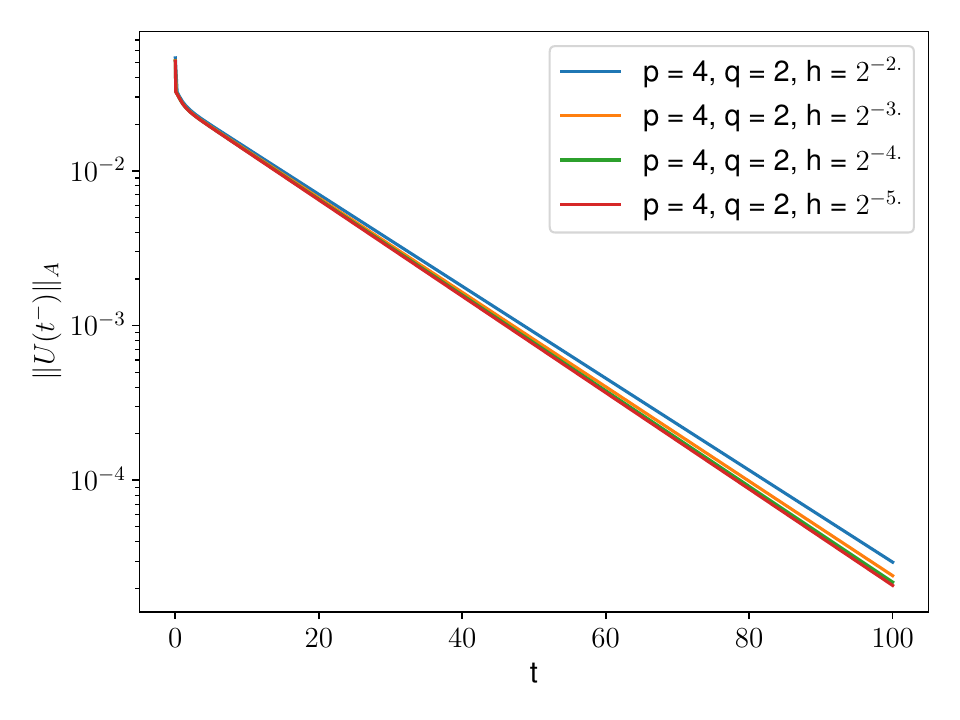}
\caption{\rev{Convergence to equilibrium for different values of the discretization parameters.}}\label{fig:convergenceToEquilibrium}
\end{figure}

The exponential convergence observed in this particular example is faster than the one postulated in \eqref{eq:convergenceToEquilibriumForFullyDiscrete}, as it does not appear to be adversely dependent on $h$ and/or $p$.
It remains an open question whether the degeneration of exponential decay for small $h$ and/or large $p$ can be witnessed in carefully constructed numerical examples and it is a direction of future research.


\bibliographystyle{siamplain}
\bibliography{ref}
\end{document}